\documentclass{article}
\usepackage{alltt}

\usepackage{setspace}
\usepackage[section]{placeins}
\usepackage{graphicx, epstopdf}

\usepackage{amsmath, amssymb, amsfonts, listings, algorithmic}
\usepackage{rotating, lscape}
\usepackage{url}

\def\ba{\begin{array}}
\def\ea{\end{array}}
\def\bea{\begin{eqnarray}}
\def\eea{\end{eqnarray}}
\def\la{\langle}
\def\ra{\rangle}

\newcommand{\ST}{{\mbox{s.t.}}}
\newcommand{\defeq}{\stackrel{\mathrm{def}}{=}}
\newcommand{\tr}{\operatorname{tr}} 
\newcommand{\Herm}{\mathcal{H}} 
\newcommand{\diag}{\operatorname{diag}} 

\newenvironment{proof}{{\noindent \textbf{Proof:}}}{\hfill\rule{2mm}{2mm}\par}

\newtheorem{lemma}{Lemma}
\newtheorem{theorem}{Theorem}
\newtheorem{proposition}{Proposition}
\newtheorem{remark}{Remark}

\title{A Subgradient Method for Free Material Design
 \thanks{Submitted to the editors .
}}

\author{
Michal Ko\v{c}vara \thanks{School of Mathematics, University of Birmingham, United Kingdom, and Institute of Information Theory and Automation, Academy of Sciences of the Czech Republic, Prague, Czech Republic (email: {kocvara@maths.bham.ac.uk}).}
\and
Yurii Nesterov \thanks{ Center of Operations Research and Econometrics (Catholic University of Louvain), Belgium and National Research University Higher School of Economics, Russia (email: { Yurii.Nesterov@uclouvain.be}).}
\and
Yu Xia \thanks{Faculty of Business Administration, Lakehead University, Canada (email: {yxia@lakeheadu.ca}). The research of this author is supported in part by a discovery grant of NSERC.}
}

\begin{document}

\maketitle
\begin{abstract}
A small improvement in the structure of the material could save the manufactory a lot of money.  The free material design can be formulated as an optimization problem.  However, due to its large scale, second-order methods cannot solve the free material design problem in reasonable size.
We formulate the free material optimization (FMO) problem into a saddle-point form in which the inverse of the stiffness matrix $A(E)$ in the constraint is eliminated.  The size of $A(E)$ is generally large, denoted as $N \times N$.  This is the first formulation of FMO without $A(E)^{-1}$.
We apply the primal-dual subgradient method \cite{Nes07} to solve the restricted saddle-point formula.  This is the first gradient-type method for FMO.
Each iteration of our algorithm takes a total of $\mathcal{O}(N^2)$ floating-point operations and an auxiliary vector storage of size $\mathcal{O}(N)$, compared with formulations having the inverse of $A(E)$ which requires $\mathcal{O}(N^3)$ arithmetic operations and an auxiliary vector storage of size $\mathcal{O}(N^2)$.
To solve the problem, we developed a closed-form solution to a semidefinite least squares problem and an efficient parameter update scheme for the gradient method, which are included in the appendix.  We also approximate a solution to the bounded Lagrangian dual problem.
The problem is decomposed into small problems each only having an unknown of $k\times k$ ($k=3$ or $6$) matrix, and can be solved in parallel.
The iteration bound of our algorithm is optimal for general subgradient scheme.
Finally we present promising numerical results.
\end{abstract}

\textbf{keywords}
	fast gradient method, Nesterov's primal-dual subgradient method, free material optimization, large-scale problems, first-order method, saddle-point, Lagrangian, complexity, duality, constrained least squares.

\textbf{AMS}
	90C90, 90C06, 90C25, 90C30, 90C47, 9008

\section{Introduction}

The approach of \textit{Free Material Optimization (FMO)} optimizes the material structure while the distribution of material and the material itself can be freely varied.  FMO has been used to improve the overall material arrangement in air frame design (\url{www.plato-n.org}).
The fundamentals of FMO were introduced in \cite{MR1327483, Ringertz}.  And the model was further developed in \cite{MR1724765, ZKB97} etc.  In the model, the elastic body of the material under consideration is represented as a bounded domain with a Lipschitzian boundary in a two- or three-dimensional Euclidean space depending on the design requirement.  For computational purpose, the domain is discretized into $m$ finite elements: $\Omega = (\Omega_1, \dots, \Omega_m)$ so that all the points in the same element are considered to have the same property.

Let $u(x)$ denote the \textit{displacement vector} of the body at point $x$ under load. 
Denote the \textit{(small-)strain tensor} as:
$$
e_{ij}\left(u(x)\right) \defeq \frac{1}{2} \left( \frac{\partial u(x)_i}{\partial x_j} + \frac{\partial u(x)_j}{\partial x_i} \right) .
$$ 
Let $\sigma_{ij}(x)$ ($i, j=1, \dots, 3$) denote the \textit{stress tensor}.
The system is assumed to follow the Hooke's law\textemdash the stress is a linear function of the strain:
$$
\sigma_{ij}(x) = E_{ijkl}(x)e_{kl}\left(u(x)\right)  \quad \text{(in tensor notation)},
$$
where $E$ is the \textit{(plain-stress) elasticity tensor} of order $4$, which maps the strain to the stress tensor.  The matrix $E$ measures the degree of deformation of a material under external loads and is a symmetric positive semidefinite matrix of order $3$ for the 2-dimensional and of order $6$ for the $3$-dimensional material design problem.
The diagonal elements of $E(x)$ measure the stiffness of the material at $x$ in the coordinate directions.  Hence the trace of $E$ is used to measure the cost (resource used) of a material in the model.

Denote by $I_k$ the identity matrix of order $k$ and $S_k^{+m}$ the direct product of $m$ cones of symmetric positive semidefinite $k\times k$-matrices:
$$
\ba{rcl} S_k^{+m} & = & \underbrace{S_k^+ \times \dots \times
S_k^+}_{m \text{ times}} \ea .
$$
For a $k \times k$ symmetric matrix $M$, let $M \succeq 0$ denote $M \in S_k^+$.

Let $E_i$ denote the elasticity tensor of order $4$ for the $i$th element $\Omega_i$:
The $E_i$'s are considered to be constant on each $\Omega_i$ but can be different for different $\Omega_i$'s and are the design variables of the FMO model:
$$\ba{rcl}
E \; = \; (E_1, \dots, E_m), & & E_i \; \succeq \; 0 , \; i = 1,
\dots, m. \ea
$$

The design problem is to find a structure that is low `cost' (the tensor $E$ having small trace) and is stable under given multiple independent loads (forces).
There are some different formulas of the FMO problem depending on the design needs.
This paper focuses on the \textit{minimum-cost FMO} problem which is to design a material structure that can withstand a whole given set of loads in the worst-case scenario and the trace of $E$ is minimal. Below we describe the model based on~\cite{KSZ08}.

The ``cost"\textemdash stiffness of the material\textemdash is measured by the trace of $E$: $\tr(E) = \sum_{i=1}^m \tr(E_i) =  \la I_k, E \ra$.  
For each $i \in \{1, \dots, m\}$, $\tr(E_i)$ is lower bounded to avoid singularity in the FMO model.
The constraints for the point-wise stiffness upper and lower bounds are:
$$ \tr{E_i} \leq \rho_u^{(i)}, \qquad   \tr{E} \geq \rho_L^{(i)}.  $$

From the engineering literature the dynamic stiffness of a structure can be improved by raising its fundamental eigenfrequency. Thus we have a lower bound on its eigen values:
$$\lambda_{\min} (E) \geq r . $$

Let $n$ be the number of nodes (vertices of the elements).
Let $nig$ denote the number of Gauss integration points in each element.
In every element, the displacement vector $u(x)$ is approximated as a continuous function which is linear in every coordinate:
$$
u(x) = \sum_{i=1}^n u_i \vartheta_i (x),
$$
where $u_i$ is the value of $u$ at the $i$th node, and $\vartheta_i$ is the basis function associated with the $i$th node.  For $\vartheta_j$, define matrices
$$
\hat{B}_j \, \defeq \, 
\begin{pmatrix} \frac{\partial \vartheta_j}{\partial x_1} & 0  \\[1ex]
0 & \frac{\partial \vartheta_j}{\partial x_2} \\[1ex]
\frac{1}{2}\frac{\partial \vartheta_j}{\partial x_2} &
\frac{1}{2}\frac{\partial \vartheta_j}{\partial x_1}
\end{pmatrix} \; (\text{for 2-dimension}), 
\qquad
\hat{B}_j \, \defeq \, 
\begin{pmatrix} \frac{\partial \vartheta_j}{\partial x_1} & 0  & 0 \\[1ex]
0 & \frac{\partial \vartheta_j}{\partial x_2} & 0 \\[1ex]
0 & 0 & \frac{\partial \vartheta_j}{\partial x_3} \\[1ex]
\frac{1}{2}\frac{\partial \vartheta_j}{\partial x_2} &
\frac{1}{2}\frac{\partial \vartheta_j}{\partial x_1} & 0 \\[1ex]
0 & \frac{1}{2}\frac{\partial \vartheta_j}{\partial x_3} &
\frac{1}{2}\frac{\partial \vartheta_j}{\partial x_2} \\[1ex]
\frac{1}{2}\frac{\partial \vartheta_j}{\partial x_3} & 0 &
\frac{1}{2}\frac{\partial \vartheta_j}{\partial x_1}
\end{pmatrix} \; (\text{for 3-dimension}). 
$$
For $\Omega_i$, let $B_{i,k}$ be the block matrix whose $j$th block is $\hat{B}_j$ evaluated at the $k$th integration point and zero otherwise.
The full dimension of $B_{i,k}$ is $3 \times 2n$ for the $2$-dimensional case and $6 \times 3n$ for the $3$-dimensional case.

Let $A(E)$ denote the \textit{stiffness matrix} relating the forces to the displacements;
Let $A_i \in \mathbb{R}^{N \times N}$ denote the \textit{element stiffness matrices}:
$$ A(E) \defeq \sum_{i=1}^m A_i(E), \qquad A_i(E) = \sum_{k=1}^{nig} B_{i,k}^\top E_i B_{i, k} . $$

Since the material obeys Hooke's law, forces (loads) on each element, denoted as $f_j \; (j=1, \dots L)$, are linearly related to the displacement vector:  
\begin{equation} \label{eq:FAu}  f_j = A(E)u \qquad j = 1, \dots L. \end{equation}
The system is in equilibrium for $u$ if outer and inner forces balance each other.  The equilibrium is measured by the \textit{compliances} of the system: the less the compliance, the more rigid the structure with respect to the loads.
The compliance can be represented as:
$$  f_j^\top u. $$
In the minimum-cost FMO model, an upper bound $\gamma > 0$ is imposed on the compliances.  Further in view of equation~\eqref{eq:FAu}, we have
$$\langle A(E)^{-1} f_j, f_j \rangle \leq \gamma, \quad j = 1, \dots L. $$

In summary, with given loads $f_j, (j=1, \dots L)$, imposed upper and lower bounds $\rho_l^{(i)}$ and $\rho_u^{(i)}$ $(i=1, \dots m)$, $r$, and compliance upper bound $\gamma$,   
 the minimum-cost multiple-load material design problem is the following:
\begin{equation} \label{MLSDP}
\ba{rl}
\min\limits_{E \in S_k^{+m}} & \sum\limits_{i=1}^m \la I_k, E_i \ra \\
\ST & \rho_l^{(i)} \; \leq \la I_k, E_i \ra \; \leq \rho_u^{(i)} ,
\; i = 1, \dots, m, \\
& \la A(E)^{-1} f_j, f_j \ra \; \leq \; \gamma, \; j = 1, \dots, L \\
&  \lambda_{\min}(E) \geq r .
 \ea
\end{equation}
Some optimization approaches have been applied to FMO; for instance,
Zowe et al.\! \cite{ZKB97} formulate the multiple-load FOM as a $\max$-$\min$ convex program.  They propose penalty/barrier multiplier methods and interior-point methods for the problem.
Ben-Tal et al.\! \cite{MR1724765} consider bounded trace minimum compliance multiple-load FMO problem.  They formulate the problem as a semidefinite program and solve the problem by an interior-point method. 
Stingl et al.\! \cite{SKL09} solve the minimum compliance multiple-load FMO problem by a sequential semidefinite programming algorithm.
Weldeyesus and Stolpe \cite{WS15} propose a primal-dual interior point method to several equivalent FMO formulations.
Stingl et al.\! \cite{SKL092} study minimum compliance single-load FMO problem with vibration constraint and propose an approach to the problem based on nonlinear semidefinite low-rank approximation of the semidefinite dual.
Haslinger et al.\! \cite{HKLS10} extend the original problem statement by a class of generic constraints.
Czarnecki and Lewi\'{n}ski \cite{CL14} deal with minimization of the weighted sum of compliances related to the non-simultaneously applied load cases.
All of them are second-order methods.  To our knowledge, no first-order methods have been employed to FMO.  

Second-order methods exploit the information of Hessians in addition to gradients and function values.  Thus, compared with first-order methods, second-order methods generally converge faster and are more accurate; on the other hand, first-order methods don't require formulation, storage, and inverse of Hessian and thus can be applied to large-scale problems. 
For certain structured problems with bounded simple feasible sets, Nesterov~\cite{MR2166537} showed that the complexity of fast gradient methods is one magnitude lower than the theoretical lower complexity bound of the gradient-type method for the black-box oracle model.  After that work, there appears quite a lot of papers on fast gradient-type methods, such as \cite{MR2112984, Nes07, NesDualExtra, Nes11, Nes13, DGN14, MR2782122, Nes15, Xia15}.

However, not every real-world problem is suitable for second-order methods or fast gradient-type methods;  for instance, when the structure of the problem is too complex to apply the interior-point method or the smoothing technique to.
The minimum weight FMO model~\eqref{MLSDP} is such a case.  For the model, although the matrices $B_{i,l}$ are sparse, $A(E)$ is generally not.  The number $m$ is at least thousands; and $n$ is smaller than $m$ only by a constant factor.  To roughly measure the amount of work per iteration, we use flops, i.e. floating point operations, such as arithmetic operations ($+, - , *, /, \sqrt{\cdot}, \frac{1}{\cdot} $), comparisons and exchanges. It takes a vector of length $\frac{N(N+1)}{2}$ to store the matrix $A(E)$ or its Cholesky factor in the memory, and about  $\big[ (k+\frac{1}{2})nig \cdot mN^2 + \frac{1}{2}N^3 \big]$ flops to evaluate $\la A(E)^{-1} f_j, f_j \ra$.  Hence, it is difficult to manage model~\eqref{MLSDP} of reasonable size by second-order methods, since second-order methods work on the Hessian of the problem whose size is at least the square of total variables.  And the variables of model~\eqref{MLSDP} are $m$ matrices of size $k \times k$.  In addition, the constraints of model~\eqref{MLSDP} are not simple, which prevents us from applying usual gradient-project type methods to it, because it is not easy to project onto its feasible set.

In this paper, we reformulate model~\eqref{MLSDP} into a saddle-point problem and apply the primal-dual subgradient method \cite{Nes07} to the saddle-point problem.  The advantage of our formulation is that the inverse or the Cholesky factorization of $A(E)$ doesn't need to be calculated; thus reduce the computational cost of each iteration to just $\mathcal{O}(N^2)$. 

The traditional subgradient method for minimizing a nonsmooth convex function $F$ over the Euclidean space employs a pre-chosen sequence of steps $\{\lambda_k\}_{k=0}^{\infty}$ which satisfies the divergent-series rule:
$$ 
\lambda_k > 0, \qquad \lambda_k \rightarrow 0, \qquad \sum_{k=0}^{\infty} \lambda_k = \infty. 
$$
The iterates are generated as follows:
$$ g_k \in \partial F(x_k), \qquad x_{k+1} = x_k - \lambda_k g_k, \qquad k \geq 0 .
$$
In the traditional subgradient method, new subgradients enter the model with decreasing weights, which contradicts to the general principle of iterative scheme\textemdash new information is more important than the old one. But the vanishing of steps is necessary for the convergence of the iterates $\{x_k\}_{k=0}^{\infty}$.

The primal-dual subgradient method~\cite{Nes07} associates the primal minimization sequence with a master process in the dual space;  it doesn't have the drawback of diminishing step sizes in the dual space;
the method is proven to be optimal for saddle-point problems, nonsmooth convex minimization, minimax problems, variational inequalities, and stochastic optimization.
Let $\mathcal{E}$ be a finite dimensional real vector space equipped with a norm $\| \cdot \|$.  Let $\mathcal{E}^*$ be its dual.  Let $Q \in \mathcal{E}$ be a closed convex set.  Let $d(x)$ be a prox-function of $Q$ with convexity parameter $\sigma \geq 0$:  $\forall \, x, y \in Q$, $\forall \alpha \in [0, 1]$:
$$
d\left( \alpha x + (1-\alpha)y \right) \leq \alpha d(x) + (1-\alpha) d(y) - \frac{1}{2} \sigma \alpha(1-\alpha) \left\| x - y \right\|^2 .
$$
Let $\mathcal{G}$ be a function mapping $\mathcal{E}$ to $\mathcal{E}^*$.  For instance, for the convex minimization problem, the function $\mathcal{G}$ can be a subgradient of the objective function.
The generic scheme of dual averaging (DA-scheme)~\cite{Nes07} works as below:

\noindent\rule{0.8\paperwidth}{.4pt}
\begin{description}
\item[\textit{Initialization:}] Set $s_0 = 0 \in \mathcal{E}^*$.  Choose $\beta_0 > 0$.

\item[\textit{Iteration}] ($k \geq 0$):
\begin{enumerate}
\item  Compute $g_k = \mathcal{G}(x_k)$.
\item
Choose $\lambda_k > 0$.  Set $s_{k+1} = s_k + \lambda_k g_k$.
\item
Choose $\beta_{k+1} \geq \beta_k$.  
Set  $ x_{k+1} = \arg \min_{x \in Q} \left\{\langle s_{k+1}, x \rangle + \beta_{k+1} d(x) \right\}$.
\end{enumerate}
\end{description}
\noindent\rule{0.8\paperwidth}{.4pt} \\

Let
$$ \hat{\beta}_0 = \hat{\beta}_1, \quad \hat{\beta}_{i+1} = \hat{\beta}_{i} + \frac{1}{\hat{\beta}_{i}} , \quad i \geq 1 .$$

The scheme has two main variants: simple averages where $\lambda_k = 1$ and $\beta_k = \gamma \hat{\beta}_k$ with constant $\gamma >0$ , and weighted averages where $\lambda_k = 1/\|g_k \|_*$ and $\beta_k = \frac{\hat{\beta}_k}{\rho \sqrt{\sigma}}$ with constant $\rho > 0$.

There are some other gradient methods for saddle-point problems.  In \cite{MR2782122}, Chambolle and Pock study a first-order primal-dual algorithm for a class of saddle-point problems in two finite-dimensional real vector spaces $\mathcal{E}$ and $\mathcal{V}$:
$$
\min_{x \in \mathcal{E}} \max_{y \in \mathcal{V}} \langle K x, y \rangle + G(x) - T^*(y),
$$
where $K \colon \mathcal{E} \rightarrow \mathcal{V}$ is a linear operator, and $G$ and $T^*$ are proper convex, lower-semicontinuous functions.  
That algorithm, as well as the classical Arrow-Hurwicz method~\cite{MR0108399} and its variants for saddle-point problems, is not applicable to our FMO formulation, because in our formulation the function between two spaces is nonlinear.
Nemirovski's prox-method~\cite{MR2112984} reduces the problem of approximating a saddle-point of a $C^{1,1}$ function to that of solving the associated variational inequality by a prox-method.  The approach is not applicable to our FMO formulation, because the structure of our FMO formulation is not simple enough and its objective function is not in $C^{1,1}$.

In our approach, the inverse of $A(E)$ in model~\eqref{MLSDP} doesn't need to be calculated, which decreases computational cost per iteration by one magnitude.
Solutions of the primal and dual subproblems at each iteration can be written in closed-form.  
Each iteration takes roughly $(6k\cdot nig \cdot L)mN$ flops.  And the auxiliary storage space is linear in $N$. 
Furthermore, since the primal subproblem is decoupled into $m$ small problems that can be solved in parallel. And each small problem can be solved in approximately $(10k^3)$ flops.
Thus, it is possible to work on large-scale problems, compared with second-order methods dealing with the Hessian of $6m$ or $21m$ variables plus additional constraints on the $m$ matrices.
To prove the efficiency of the algorithm, we give iteration complexity bounds of our algorithm, which includes \textit{simple dual averaging} and \textit{weighted dual averaging} schemes.  The complexity bounds are optimal for the general subgradient methods.
Numerical experiments are described at the end of the paper.

The remainder of the paper is organized as follows.  In Section \ref{S:saddle}, we give our saddle-point form of the problem. 
In Section \ref{S:bounds}, we show that a solution to our bounded Lagrangian form either solves the original problem or gives an approximate solution of the original problem.
In Section \ref{S:algo}, we present our algorithm. 
In Section \ref{S:subsolution}, we give closed-form solutions to the subproblems.
In Section \ref{S:boundd}, we derive complexity bounds of our algorithm.  In section \ref{S:computing}, we present some computational examples of our algorithm.
In Section \ref{S:penalty}, we describe and analyze a penalized lagrangian approach.
In the Appendixes, we give a closed-form solution of a related matrix projection problem and an update scheme for the parameters of the algorithm.

\section{Saddle-Point Formulation} \label{S:saddle}
We first rewrite problem \eqref{MLSDP} in a saddle-point form.
Denote
\begin{equation} \label{def:Qk}
 Q_k^{(i)} \defeq \left\{ U \in S_k \colon \lambda_{\min}(U) \geq
r, \; \rho_l^{(i)} \leq \la I_k, U \ra \leq \rho_u^{(i)} \right\} . 
\end{equation}

The second group of constraints in \eqref{MLSDP} can be represented in max form:
$$
\ba{rcl}
\gamma \; \geq \; \la A(E)^{-1} f_j, f_j \ra & = & \max\limits_{y_j \in \mathbb{R}^N} \left\{ 2\la f_j, y_j \ra  - \la A(E) y_j, y_j \ra \right\} .
\ea
$$

Assume that problem~\eqref{MLSDP} satisfies some constraint qualifications, such as the Slater condition\textemdash there exists $\hat{E}$ such that $\langle A(\hat{E})^{-1} f_j, f_j \rangle < \gamma$ for $j=1, \dots L$.   Then a Lagrangian multiplier exists; and we can solve the Lagrangian of problem~\eqref{MLSDP} instead.  Thus, problem \eqref{MLSDP} can be written as follows:
$$
\ba{rcl}
& & \min\limits_{\stackrel{E_i \in Q_k^{(i)}}{ k=1, \dots, m}} \max\limits_{\stackrel{y_j \in \mathbb{R}^N, \lambda_j \geq 0}{j=1, \dots, L}}
\left\{
\sum\limits_{i=1}^m \la I_k, E_i \ra  + \sum\limits_{j=1}^L  \lambda_j \cdot \left[ 2 \la f_j, y_j \ra - \la A(E) y_j, y_j \ra - \gamma \right] \right\} \\
\\
& \stackrel{\lambda_j y_j \rightarrow x_j}{=} &
\min\limits_{\stackrel{E_i \in Q_k^{(i)}}{ k=1, \dots, m}}  \left\{
\sum\limits_{i=1}^m \la I_k, E_i \ra  +  \max \left( 0, 
\max\limits_{\stackrel{x_j \in \mathbb{R}^N, \lambda_j > 0}{j=1, \dots, L}}
\sum\limits_{j=1}^L  \left[ 2 \la f_j, x_j \ra - \frac{1}{\lambda_j} \la A(E) x_j, x_j \ra - \gamma \lambda_j \right] \right) \right\} \\
\\
& = & 
\min\limits_{\stackrel{E_i \in Q_k^{(i)}}{ k=1, \dots, m}}
\max\limits_{\stackrel{x_j \in \mathbb{R}^N,}{j=1, \dots, L}}
\left\{
\sum\limits_{i=1}^m \la I_k, E_i \ra  + \sum\limits_{j=1}^L  2 \left[ \la f_j, x_j \ra - \gamma^{1/2} \la A(E) x_j, x_j \ra^{1/2} \right] \right\} .
\ea
$$
The dimension of the matrix $A(E)$ is large; the first transformation eliminates the need of calculating its inverse, but that results in a nonconcave objective function in $\lambda$ and $y$.  The second transformation makes the function concave in $\lambda$ and $x$.  In the last step, variable $\lambda$ is eliminated to simplify the formulation.

Denote
$$
\ba{rcl} F(E, x) & \defeq & \sum\limits_{i=1}^m \la I_k, E_i \ra + \sum\limits_{j=1}^L 2 \left[ \la f_j , x_j \ra - \gamma^{1/2} \la A(E) x_j, x_j \ra^{1/2} \right]. \ea
$$
Thus, to solve problem~\eqref{MLSDP}, we only need to solve
\begin{equation} \label{UBSPF}
\min\limits_{\stackrel{E_i \in Q_k^{(i)}}{i=1, \dots, m}} \max\limits_{\stackrel{x_j \in \mathbb{R}^N,}{j=1, \dots, L} } F(E, x).
\end{equation}
Note that $F(E,x)$ is convex in $E$ and concave in $x \in \mathbb{R}^{L \times N}$.

\section{Bounded Lagrangian} \label{S:bounds}
We apply the primal-dual subgradient method~\cite{Nes07} to the saddle-point formulation \eqref{UBSPF}.
The convergence of the algorithm requires the iterates be uniformly bounded \cite{Nes07}.  We therefore impose a bound on $x$: 

\begin{equation} \label{BSPF}
\min\limits_{\stackrel{E_i \in Q_k^{(i)}}{i=1, \dots, m}} \max\limits_{\stackrel{\|x_j\| \leq \eta,}{j=1, \dots, L} } F(E, x).
\end{equation}
Next we show that the primal solution of the saddle-point problem~\eqref{BSPF} is either a solution to the original problem \eqref{MLSDP} or an approximate solution in the sense that its constraint-violation is bounded by $\eta^{-1}$ and its objective value is smaller than that of the optimal value of \eqref{MLSDP}.

Let $(E^*, x^*)$ be a solution to the saddle-point problem \eqref{UBSPF}; then for any $\alpha \geq 0$, $(E^*, \alpha x^*)$ is also its solution.
We can choose $\alpha^*$ small enough; for instance, let $\alpha^* = \eta/ \max\{ \|x_j^*\|, 1 \}$, so that $(E^*, \alpha^* x^*)$ is a solution to the bounded saddle-point form \eqref{BSPF}.

For any $E \in Q_k$, denote the index set of its violated constraints as
$$
\ba{rcl}
W_E & \defeq & \left\{ 1 \leq j \leq L \colon \la f_j, A(E)^{-1} f_j \ra > \gamma  \right\} .
\ea
$$
Denote $F(E) \defeq \max_x F(E, x)$.
We have the following results regarding our material design problem.

\begin{lemma} \label{lm:bdL}
Let $(\tilde{E}, \tilde{x})$ be a solution of \eqref{BSPF}. Let $F^*$ be the optimal value of \eqref{MLSDP}.
\begin{enumerate}
\item \label{item:EInt}
If $\|\tilde{x}_j\| < \eta$ for $j=1, \dots, L$; then $\tilde{E}$ is a solution to \eqref{MLSDP}.
\item \label{item:EBd}
Otherwise, $\tilde{E}$ has the following properties:
\begin{enumerate}
\item \label{item:ObjLess}
$ F(\tilde{E}) \leq F^*$.
\item \label{item:ConstrainBound}
$\sum\limits_{j \in W_{\tilde{E}}} \left( \la f_j, A(\tilde{E})^{-1} f_j \ra^{1/2} - \gamma^{1/2} \right) \leq \frac{F^* - m\rho_l}{2r \lambda_{\min}(BB^T)\eta}$ .
\end{enumerate}
\end{enumerate}
\end{lemma}

\begin{proof}
Item~\ref{item:EInt} is obvious as the constraints are non-binding.

Next, we prove item~\ref{item:EBd}.

Because
$$
\max\limits_{\stackrel{ \|x_j\| \leq \eta}{j=1, \dots, L} } F(E, x) \leq
\max\limits_{\stackrel{x_j \in \mathbb{R}^N,}{j=1, \dots, L} } F(E, x), 
$$
we have item (\ref{item:ObjLess}).
For any fixed $E \in Q$, the point
$$
\ba{rcl}
x_j & = & \begin{cases} \frac{\eta}{\|A(E)^{-1}f_j\|}A(E)^{-1}f_j & j \in W_E \\
0 & j \notin W_E \end{cases}
\ea
$$
is feasible to
$$
\max\limits_{\stackrel{\|x_j\| \leq \eta,}{ j=1, \dots, L} }F(E, x)
$$
with objective value
\begin{equation} \label{optx}
F_x(E)  \defeq  \la I, E \ra + 2\sum\limits_{j \in W_E} \left( \la f_j, A(E)^{-1} f_j \ra^{1/2} - \gamma^{1/2} \right) \frac{\la f_j, A(E)^{-1} f_j \ra^{1/2} }{\|A(E)^{-1}f_j\|} \eta .
\end{equation}

Since
$$
\ba{rcl}
\la f_j, A(E)^{-1} f_j \ra & = & \la A(E)^{-1}f_j, A(E)(A(E)^{-1} f_j) \ra ,
\ea
$$

we also have
$$
\ba{rcl}
\frac{\la f_j, A(E)^{-1} f_j \ra^{1/2} }{\|A(E)^{-1}f_j\|}  & \geq & \lambda_{\min} A(E)^{1/2} \geq r \lambda_{\min}(BB^T) , \\
\text{and}\\
\la I, E \ra & \geq & m \rho_l .
\ea
$$

Therefore,
$$
\ba{rcl}
F(\tilde{E}, \tilde{x}) & \geq & m \rho_l + 2r\lambda_{\min}(BB^T) \eta \sum_{j \in W_{\tilde{E}}} \left( \la f_j, A(\tilde{E})^{-1} f_j \ra^{1/2} - \gamma^{1/2} \right) ,
\ea
$$

By item~(\ref{item:ObjLess}) of the lemma, we have
$$
\ba{rcl}
F(\tilde{E}, \tilde{x}) & \leq & F^* .
\ea
$$
Item~(\ref{item:ConstrainBound}) then follows.
\end{proof}
Note that as $\eta \rightarrow + \infty$, the set of saddle-points of \eqref{BSPF} approaches that of the original problem.

\section{The Algorithm} \label{S:algo}
In this part, we describe how to apply the primal-dual subgradient method~\cite{Nes07} to the saddle-point reformulation of model~\eqref{MLSDP}.
We have developed a parameter update scheme for the algorithm, which is included in the Appendix.

For a matrix $V$, let vector $\lambda(V)$ denote the eigenvalues of $V$;
let $\lambda_{\min}(V)$ be the smallest eigenvalue of $V$.
The gradient (subgradients) of $F(E, x)$ at $(E, x)$ are: for $i=1, \dots, m$, $ j=1. \dots, L$,
$$
\ba{rcl}
g_{E_i}(E, x) & = &
 I_k - \sqrt{\gamma} \sum\limits_{j \in R} \la A(E)x_j, x_j \ra^{-1/2}\left(\sum\limits_{l=1}^{nig} B_{i,l}x_jx_j^TB_{i,l}^T\right), \\
 & & \text{where }R \defeq \left\{ 1 \leq l \leq L \colon \la A(E)x_l, x_l \ra > 0\right\} ;\\
\\
g_{x_j}(E, x) & =  &
\begin{cases} 2f_j - 2\sqrt{\gamma} \la A(E)x_j, x_j \ra^{-1/2}A(E)x_j  & \la A(E)x_j, x_j \ra > 0 \\
\left\{2f_j - 2\sqrt{\gamma}A(E)y \colon \la A(E)y, y \ra = 1 \right\} & \la A(E)x_j, x_j \ra = 0 .
\end{cases}
\ea
$$
For the primal space, we choose the standard Frobenius norm:
$$
\|E\|_F^2 \, = \, \sum_{i=1}^m \|E_i\|_F^2 \, = \, \tr(E^2), \qquad
d(E) \, = \, \frac{1}{2} \sum_{i=1}^m \|E_i - rI_k \|_F^2 .
$$
For the dual space, we choose the standard Euclidean norm:
$$
\|x\|_2^2 \, = \, \sum_{j=1}^L \|x_j\|_2^2 \, = \, x^Tx, \qquad d(x)
\, = \, \frac{1}{2} \sum_{j=1}^L \|x_j\|^2_2 .
$$

Their dual norms are denoted as:
$\| \cdot \|_{F,*} = \| \cdot \|_F$ , $\| \cdot \|_{2,*} = \| \cdot \|_2$ .

The set $Q_k^{(i)}$ for $E_i$ is defined in \eqref{def:Qk}; and the set $Q_x$ for $x_j$ is
$$
\ba{rcl}
Q_x & = & \left\{ x_j \in \mathbb{R}^N \colon \|x_j\| \leq \eta \right\} . 
\ea
$$

Note that $F$ is nonsmooth.  The primal-dual subgradient method~\cite{Nes07} for saddle-point problems~\eqref{BSPF} works as follows.\\

\begin{tabular}{l} \hline
Initialization: Set $s_0^{E_i} = 0 \, (i=1, \dots, m)$ , $s_0^{x_j} = 0 \, (j=1, \dots, L)$. \\
\qquad Choose $\beta_0 > 0$, $0 < \tau < 1$. \\ \hline
Iteration $t = 0, 1, \dots$ \\
1. Compute $g_{E_i}^{(t)}(E^{(t)}, x^{(t)})$, $g_{x_j}^{(t)}(E^{(t)}, x^{(t)})$, for $i=1, \dots, m; j=1, \dots, L$. \\
2. Choose $\alpha_t > 0$,
set \\
\begin{tabular}{c}
$
s_{t+1}^{E_i} = s_t^{E_i} + \alpha_t g_{E_i}^{(t)} \, (i=1, \dots, m) , \quad
s_{t+1}^{x_j} = s_t^{x_j} - \alpha_t g_{x_j}^{(t)} \, (j=1, \dots, L) .
$
\end{tabular}
\\
3. Choose $\beta_{t+1} \geq \beta_t $, set \\
\begin{tabular}{c}
$
E^{(t+1)} = \arg \min\limits_{E_i \in Q_k^{(i)}} \left\{ \la s_{t+1}^E, E \ra + \beta_{t+1}\tau d_E(E) \right\} ,
$
\end{tabular} \\
\begin{tabular}{c}
$
x^{(t+1)} = \arg \min\limits_{x_j \in Q_{x}} \left\{ \la s_{t+1}^x, x \ra + \beta_{t+1}(1-\tau) d_{x}(x) \right\} .
$
\end{tabular}
\\
\hline
Output: $\hat{E}^{(t+1)} = \frac{1}{\sum_{l=0}^t \alpha_l} \sum_{l=0}^t \alpha_l E^{(l)}$.\\
\hline
\end{tabular}
\\
Details of a parameter update scheme for $\beta_t$ is given in the Appendix.

We take
$$
\ba{c}
\hat{\beta}_0 \; = \; \hat{\beta}_1 \; = \; 1, \quad  \hat{\beta}_{t+1} \; = \; \hat{\beta}_t + \frac{1}{\hat{\beta}_t}, \; t \; = \; 1, \dots \\
\\
\beta_t \; = \; \sigma \hat{\beta}_t , \; t \; =  \; 0, \dots .
\ea
$$

Based on different choices of $\alpha$, there are two variants of the algorithm:
\begin{enumerate}
\item{Method of Simple Dual Averages}

We let
$$
\ba{rcl}
\alpha_t & = & 1 , \;  t = 0, \dots .
\ea
$$

\item{Method of Weighted Dual Averages}

We let
$$
\ba{rcl}
\alpha_t & = & 1/ \left( \frac{\|g_E^{(t)}\|_{F,*}^2}{\tau} + \frac{\|g_x^{(t)}\|_{2,*}^2}{1-\tau}  \right)^{1/2} , \; t = 0, \dots .
\ea
$$

\end{enumerate}

\section{Solution to the Subproblem} \label{S:subsolution}
In this part, we give closed-form solutions to the subproblems at each iteration of our algorithm.

\paragraph{Solution of $x$.}
The closed-form solution for $x^{(t+1)}$ in \textit{Step 3} of the algorithm is derived as below.

By Cauchy-Schwartz-Boniakovsky inequality, for $j=1, \dots, L$:

$$
\ba{rcl}
&& \la s_{t+1}^{x_j}, x_j \ra + \beta_{t+1}(1-\tau) d_{x_j}(x_j) \\
\\
& \geq & - \|s_{t+1}^{x_j}\|_{2,*} \cdot \|x_j\|_2 + \frac{\beta_{t+1}(1-\tau)}{2}\|x_j\|_2^2 \\
\\
& = & \frac{\beta_{t+1}(1-\tau)}{2}\left( \|x_j\|_2 -
\frac{1}{\beta_{t+1}(1-\tau)}\|s_{t+1}^{x_j}\|_{2,*} \right)^2 -
\frac{1}{2\beta_{t+1}(1-\tau)}\|s_{t+1}^{x_j}\|_{2,*}^2 , \ea
$$
with equality if{f} $x_j = - \nu s_{t+1}^{x_j}$ for some $\nu \geq 0$.  Therefore,
\begin{equation} \label{eq:solutionx}
x^{(t+1)}_j  = 
- \min \left( \frac{ \eta }{\|s_{t+1}^{x_j}\|_{2,*}}, \frac{1}{\beta_{t+1}(1-\tau)} \right) s_{t+1}^{x_j} .
\end{equation}

\paragraph{Solution of $E$.}
For a set $M$, let $|M|$ denote the cardinality of $M$; i.e., the number of elements in $M$.
In \textit{Step 3} of the algorithm, $E^{(t+1)}_i$ can be seen as the projection
$$
\min\limits_{V \in Q_k^{(i)}} \quad \| V + \frac{1}{\beta_{t+1}\tau} s_{t+1}^{E_i} - rI \|_F^2 .
$$

By Theorem~\ref{thrm:projH} in \textit{Appendix: Matrix Projection}, we can represent $E^{t+1}$ as follows. 

For each $1 \leq i \leq m$, let $U \Lambda U^T$ be the eigenvalue
decomposition of $s_{t+1}^{E_i}$,  and $\lambda_1, \dots, \lambda_k$
be its eigenvalues.  Define the sets
$$
M_0 \,  \defeq \, \{ 1 \leq l \leq k \colon \lambda_l \geq 0 \}, \quad
\bar{M}_0 \, \defeq \, \{1, \dots, k \} \setminus M_0 .
$$
Then
\begin{equation} \label{eq:solutionE}
\ba{rcl}
E^{(t+1)}_i & = & U \diag(\omega)U^T,
\ea
\end{equation}
where $\omega$ is determined according to the following three cases.
\begin{enumerate}
\item
$\beta_{t+1}\tau(kr-\rho_u^{(i)}) \leq \sum_{q \in \bar{M}_0} \lambda_q
\leq \beta_{t+1} \tau(kr-\rho_l^{(i)})$.

Let
$$
\ba{rcl}
\omega_l & = & \begin{cases}
r & l \in M_0 \\
r - \frac{\lambda_l}{\beta_{t+1}\tau} & l \notin M_0 .
\end{cases}
\ea
$$

\item \label{itm:lambdal}
$\sum_{q \in \bar{M}_0} \lambda_q < \beta_{t+1}\tau(kr-\rho_u^{(i)})$.

Then there is a partition $\bar{M}_0 = P \cup \bar{P}$:
$$
\ba{rcl}
P & = & \left\{ l \in \bar{M}_0 \colon \lambda_l < \frac{\beta_{t+1}\tau(\rho_u^{(i)} - kr) + \sum_{q\in P} \lambda_q}{|P|} \right\} , \\
\\
\bar{P} & = & \left\{ l \in \bar{M}_0 \colon \lambda_l \geq
\frac{\beta_{t+1} \tau ( \rho_u^{(i)} - kr) + \sum_{q\in P}
\lambda_q}{|P|} \right\} . \ea
$$
Let
$$
\omega_l = \begin{cases}
r & l \in \bar{P}\cup M_0 \\
r - \frac{\lambda_l}{\beta_{t+1}\tau} + \frac{\beta_{t+1}
\tau(\rho_u^{(i)} - kr)+\sum_{q\in P} \lambda_q}{\beta_{t+1}\tau|P|} & l \in
P .
\end{cases}
$$

\item $ \sum_{q \in \bar{M}_0} \lambda_q > \beta_{t+1}\tau(kr-\rho_l^{(i)})$. \label{itm:lambdag}

Then there is a partition $M_0 = P_m \cup \bar{P}_m$:
$$
\ba{rcl} P_m & = & \left\{ l \in M_0 \colon \frac{\rho_l^{(i)} +
\frac{1}{\beta_{t+1}\tau}  \sum_{j\in P_m \cup \bar{M}_0} \lambda_j
- kr}{|P_m| + |\bar{M}_0|} > \frac{\lambda_l}{\beta_{t+1}\tau}
\right\} . \ea
$$

Let
$$
\ba{rcl}
\omega_l & = & \begin{cases}
 - \frac{\lambda_l}{\beta_{t+1}\tau} + \frac{\beta_{t+1}\tau(\rho_l^{(i)} - |\bar{P}_m|r) + \sum_{q\in \bar{M}_0 \cup P_m} \lambda_q}{\beta_{t+1}\tau(|\bar{M}_0|+|P_m|)},
& l \in \bar{M}_0 \cup P_m. \\
r & l \in \bar{P}_m
\end{cases}
\ea
$$
\end{enumerate}
The eigenvalues $\omega$ in case~\ref{itm:lambdal} can be obtained
by the following algorithm:

\paragraph{Algorithm projSyml}
\begin{description}
\item[Step 1] (Initialization)
Let $\lambda_{\sigma(1)} \leq \dots \leq \lambda_{\sigma(p)} < 0$ be the $p$ negative eigenvalues of $s_{t+1}^{E_i}$.\\
Let
$$
P = \{ \sigma(1) \}, \quad T = \beta_{t+1}\tau(\rho_u^{(i)} - kr) + \lambda_{\sigma(1)}, \quad q = 1 .
$$
\item[Step 2] While  $q \lambda_{\sigma(q+1)} < T$, do
$$
P \cup \{\sigma(q+1)\} \rightarrow P, \quad T +
\lambda_{\sigma(q+1)}\rightarrow T, \quad q+1 \rightarrow q .
$$
\item[Step 3] Let
$$
\ba{rcl}
\omega_l & = & \begin{cases} r & l \notin P \\
r - \frac{\lambda_l}{\beta_{t+1}\tau} + \frac{T}{\beta_{t+1}\tau q} & l \in P .
\end{cases}
\ea
$$
\end{description}

\medskip

Similarly, the eigenvalues in case~\ref{itm:lambdag} can be obtained by the following algorithm:
\paragraph{Algorithm projSymg}
\begin{description}
\item[Step 1] (Initialization)
Let $0 < \lambda_{\sigma(1)} \leq \dots \leq \lambda_{\sigma(u)}$ be the $u$ positive eigenvalues of $s_{t+1}^{E_i}$.\\
\begin{itemize}
\item
If $u=p$, let 
$$
U = \{ \sigma(1) \}, \quad T = \beta_{t+1}\tau(\rho_l^{(i)} - kr) + \lambda_{\sigma(1)}, \quad q = 1 .
$$
\item
If $u < p$, let 
$$
U = \bar{M}_0 \cup \{i \colon \lambda_i =0 \} , \quad T = \beta_{t+1}\tau(\rho_l^{(i)} - kr) + \sum_{j \in \bar{M}_0} \lambda_j , \quad q = |U| .
$$
\end{itemize}
\item[Step 2] While  $ q \lambda_{\sigma(q+1)} < T$, do
$$
U \cup \{\sigma(q+1)\} \rightarrow U, \quad T +
\lambda_{\sigma(q+1)}\rightarrow T, \quad q+1 \rightarrow q .
$$
\item[Step 3] Let
$$
\ba{rcl}
\omega_l & = & \begin{cases} r & l \in  M_0 \setminus U \\
r - \frac{\lambda_l}{\beta_{t+1}\tau} + \frac{T}{\beta_{t+1}\tau q} & l \in  \bar{M}_0 \cup U .
\end{cases}
\ea
$$
\end{description}
\medskip

\section{Complexity of the Algorithm} \label{S:boundd}
To understand the complexity of the algorithm for model~\eqref{MLSDP}, in this part we study duality gap and computational cost of each iteration.
By \cite{Nes07}, it takes $\mathcal{O}(\frac{1}{\epsilon^2})$ iterations to solve a general convex-concave saddle-point problem to the absolute accuracy $\epsilon$, which is the exact lower complexity bound for such class of algorithm schemes.
To give an insight of how the data of FMO model, such as $f$, $B$, and $\eta$, affect convergence time,  
we give upper bounds of the duality gap of the iterates generated by our algorithm in terms of the number of iterations and input data in \S\S~\ref{SS:IterationBound}.
In~\S\S~\ref{SS:ComputationCost}, we derive computational cost per iteration.
From the duality gap and computational cost per iteration given in this section, we can estimate from given data how much computational effort is needed at most to approximate a solution of a problem instance of model~\eqref{MLSDP} based on the method proposed in the paper.

\subsection{Iteration Bounds} \label{SS:IterationBound}
By \cite[Chapter 6, Proposition 2.1]{ET99},
for a function $L \colon \mathcal{A} \times \mathcal{B}\mapsto \mathbb{R}$, assume 
\begin{itemize}
\setlength\itemsep{-0.3em}
\item
the sets $\mathcal{A}$ and $\mathcal{B}$ are convex, closed, non-empty, and bounded;
\item
for any fixed $u \in \mathcal{A}$, $p \mapsto L(u,p)$ is concave and upper semicontinuous;
\item
for any fixed $p \in \mathcal{A}$, $u \mapsto L(u,p)$ is convex and upper semicontinuous;
\end{itemize}
then the function $L$ has at least one saddle-point.

Since $E$ and $x$ are bounded, and $F$ is continuous and finite, by the above results, we conclude that $F$ has a saddle-point and a finite saddle-value.
An upper bound on duality gap is given in~\cite[Theorem 6]{Nes07}.
We next represent the duality gap in terms of input data.

Define
$$
\ba{rcl}
\|(g_E, g_x)\|_* & \defeq & \left[ \frac{1}{\tau}\|g_E\|_{F,*}^2 + \frac{1}{1-\tau}\|g_x\|_{2,*}^2 \right]^{1/2} ,
\ea
$$

$$
\ba{rcl}
\|(E, x)\| & \defeq & \left[ \tau \|E\|_F^2 + (1-\tau)\|x\|_2^2 \right]^{1/2} .
\ea
$$

For a matrix $V$, denote $\|V\|^2_2 \defeq \lambda_{\max}(V^TV)$.

Since
$$
\ba{rcl}
\begin{pmatrix} \rho_u - (k-1)r & & & \\ & r & & \\ & & \ddots & \\ & & & r \end{pmatrix} & \in &
\arg \max\limits_{E_i \in Q_k^{(i)}} d_E(E_i) ,
\ea
$$
we have
\begin{equation} \label{bd:DE}
\ba{rcl}
D_E & \defeq & \frac{1}{2} \max\limits_{E_i \in Q_k^{(i)}}\| E - rI \|_F^2 \; \leq \; \frac{1}{2} m(\rho_u - kr)^2 .
\ea
\end{equation}

Furthermore, by our algorithm scheme,
\begin{equation} \label{bd:SDADx}
\ba{rcl}
D_x & \defeq & \max\limits_{x \in Q_x} \frac{1}{2}\|x\|^2_2 \; \leq \; \frac{L}{2} \eta^2  .
\ea
\end{equation}

Define
$$
\ba{rcl}
\kappa_t & \defeq & \frac{1}{\sum_{l=0}^t\alpha_l} \max\limits_{E_i \in Q_k^{(i)}} \left\{\sum_{l=0}^t \alpha_l \la g_E\left(E^{(l)}, x^{(l)} \right), E^{(l)} - E \ra \right\} , \\
\\
\upsilon_t & \defeq & \frac{1}{\sum_{l=0}^t\alpha_l} \max\limits_{x_j \in \mathbb{R}^N} \left\{\sum_{l=0}^t \alpha_l \la g_x\left(E^{(l)}, x^{(l)} \right), x- x^{(l)}  \ra \colon \|x_j\|_2 \leq \eta \right\} .
\ea
$$

By Cauchy-Schwartz-Boniakovsky inequality, it is easy to verify that
$$
\ba{rcl} 
\upsilon_t & = & \frac{1}{\sum_{l=0}^t\alpha_l} \left( \sum_{j=1}^L
\eta \|s_{t+1}^{x_j}\|_{2,*}  - \sum_{l=0}^t \alpha_l \la
g_x\left(E^{(l)}, x^{(l)} \right), x^{(l)}  \ra \right) , 
\ea 
$$
which is attained at
$$
\ba{rcl}
x_j & = & \begin{cases}
 \frac{\eta}{\| s_{t+1}^{x_j}\|_{2,*}} s_{t+1}^{x_j} & s_{t+1}^{x_j} \neq 0 , \\
0 & \text{otherwise} .  
\end{cases}
\ea
$$

Now let us give a bound for $\kappa_t$.   Let $\kappa_t =
\frac{1}{\sum_{l=1}^t \alpha_l} \sum_{i=1}^m \kappa^t_i$, where
$$
\ba{rcl} \kappa^t_i & \defeq & \max\limits_{E_i \in Q_k^{(i)}}
\left\{\sum_{l=0}^t \alpha_l \la g_{E_i}\left(E^{(l)}, x^{(l)}
\right), E^{(l)}_i - E_i \ra \right\}. \ea
$$

By Hoffman-Wielandt theorem,

$$
\ba{rcl}
\kappa^t_i & = &
\sum_{l=0}^t \alpha_l \la g_{E_i}(E^{(l)}, x^{(l)}), E^{l}_i \ra
- \min\limits_{E_i \in Q_k^{(i)}} \la s_{t+1}^{E_i}, E_i \ra
\\
\\
 & = & \sum_{l=0}^t \alpha_l \la g_{E_i}(E^{(l)}, x^{(l)}), E^{l}_i \ra \\
 \\ & & -
\begin{cases}
 [\rho_l - kr]_+ \lambda_{\min}(s_{t+1}^{E_i}) + r \tr( s_{t+1}^{E_i}) & \lambda_{\min}(s_{t+1}^{E_i}) > 0 , \\
 (\rho_u - kr) \lambda_{\min}(s_{t+1}^{E_i})+ r \tr(s_{t+1}^{E_i}) & \lambda_{\min}(s_{t+1}^{E_i}) \leq 0 . \end{cases}
\ea
$$

Define
$$
\ba{rcl}
\delta_t & \defeq & \max_{E_i \in Q_k^{(i)}, x \in Q_x} \left\{\sum_{l=0}^t \alpha_l \la \left(g_E^{(l)}, g_x^{(l)}\right), (E^{(l)}, x^{(l)}) - (E, x) \ra \colon d(x) \; \leq \; \tau D_E + (1-\tau) D_x \right\}
\ea
$$

By \cite[Theorem 6]{Nes07}, $\kappa_t + \upsilon_t$ is a bound of the duality gap; i.e.
\begin{equation} \label{bd:gapt}
\ba{rcl} 0 \; \leq \; \max\limits_{ x_j \in Q_x } F(\hat{E}^{(t+1)}, x) - \min\limits_{E_i^{(i)}
\in Q_k^{(i)}} F(E, \hat{x}^{(t+1)}) & \leq & \kappa_t + \upsilon_t \; \leq \; \frac{1}{\sum_{l=0}^t\alpha_l} \delta_t . 
\ea
\end{equation}

Next, we bound the above duality gap by input data.  To this end, we first bound the partial derivatives $g_E$ and $g_x$.

Denote
\begin{equation} \label{def:B}
\ba{rcl}
B_i \; \defeq \; \begin{bmatrix} B_{i,1} \\ \vdots \\ B_{i, nig} \end{bmatrix},
& \qquad & B \; \defeq \; \begin{bmatrix} B_1 \\ \vdots \\ B_m \end{bmatrix} .
\ea
\end{equation}
\begin{lemma}
The partial derivative of $F(E, x)$ in $E$ can be bounded by $\|x\|_2$ as follows:
$$
\ba{rcl} \left\| g_E(E, x) \right\|_{F,*}^2 & \leq & mk + L^2 \frac{\gamma}{r} \|B\|_2^2 \eta^2  \; \defeq \; L_E^2 .  \ea
$$
\end{lemma}
\begin{proof}
We have
\begin{equation} \label{eq:xAEx}
\ba{rcl}
\la A(E)x_j, x_j \ra & = & \sum_{i=1}^m  \sum_{l=1}^{nig} \la E_i B_{i, l}x_j, B_{i, l}x_j \ra \\
\\
& \geq & \sum_{i=1}^m \lambda_{min}(E_i) \sum_{l=1}^{nig}\la  B_{i, l}x_j, B_{i, l}x_j \ra \\
\\
& \geq & r \sum_{i=1}^m \sum_{l=1}^{nig}\la  B_{i, l}x_j, B_{i, l}x_j \ra ,
\ea
\end{equation}
where the last inequality is from the definition of the set $Q_k^{(i)}$.

Since for two matrices $A$ and $B$ of proper dimensions, $\tr(AB) = \tr(BA)$, we have
\begin{equation} \label{eq:BxxB}
\ba{rcl}
\left\|\sum_{l=1}^{nig} B_{i,l} x_j x_j^T B_{i, l}^T \right\|_F & \leq & \sum_{l=1}^{nig} \left\|B_{i,l} x_j x_j^T B_{i, l}^T \right\|_F \\
\\
& = & \sum_{l=1}^{nig} \left[ \tr \left( B_{i,l} x_j x_j^T B_{i, l}^T B_{i,l} x_j x_j^T B_{i, l}^T \right) \right]^{1/2} \\
\\
& = & \sum_{l=1}^{nig} \left[ \left( x_j^T B_{i, l}^T B_{i,l} x_j \right)^2 \right]^{1/2} \\
\\
& = & \sum_{l=1}^{nig} \la  B_{i, l}x_j, B_{i, l}x_j \ra .
\ea
\end{equation}

Therefore,
$$
\ba{rcl}
\left\| \la A(E)x_j, x_j \ra^{-1/2}\left(\sum_{l=1}^{nig} B_{i,l}x_jx_j^TB_{i,l}^T\right) \right\|_F & \leq &
\frac{1}{\sqrt{r}} \left[ \sum_{l=1}^{nig} \la  B_{i, l}x_j, B_{i, l}x_j \ra \right]^{1/2} \\
\\ & \leq & \frac{1}{\sqrt{r}}  \lambda_{\max}\left( \sum_{l=1}^{nig} B_{i, l}^T B_{i, l} \right)^{1/2}\|x_j\|_2 .
\ea
$$
Note that
$$
\ba{rcl}
\lim\limits_{x_j \rightarrow 0} \left\| \la A(E)x_j, x_j \ra^{-1/2}\left(\sum_{l=1}^{nig} B_{i,l}x_jx_j^TB_{i,l}^T\right) \right\|_F  & = & 0 .
\ea
$$

We also have
$$
\ba{rcl}
\sum_{i=1}^m \sum_{l=1}^{nig} B_{i, l}^T B_{i, l} \; = \; B^T B , & \qquad &
\lambda_{\max}\left( \sum_{i=1}^m \sum_{l=1}^{nig} B_{i, l}^T B_{i, l} \right)^{1/2} \; =  \; \|B\|_2 .
\ea
$$

Hence, $g_{E}(E, x)$ is bounded as below:
\begin{equation} \label{bd:gE}
\ba{rcl}
\left\| g_{E}(E, x) \right\|_{F,*}^2 & \leq & \|I\|_F^2  + \gamma \sum\limits_{i=1}^m \left\| \sum\limits_{j \in R} \la A(E)x_j, x_j \ra^{-1/2} \left( \sum\limits_{i=1}^{nig} B_{i, l} x_j x_j^T B_{i, l}^T \right) \right\|_F^2 \\
\\
& \leq & \|I\|_F^2 + L \gamma \sum\limits_{i=1}^m  \sum\limits_{j \in R} \left\| \la A(E)x_j, x_j \ra^{-1/2} \left( \sum\limits_{i=1}^{nig} B_{i, l} x_j x_j^T B_{i, l}^T \right) \right\|_F^2 \\
\\
& = & mk + L \gamma \sum\limits_{j \in R}  \la A(E)x_j, x_j \ra^{-1}  \sum\limits_{i=1}^m  \left\|\sum\limits_{i=1}^{nig} B_{i, l} x_j x_j^T B_{i, l}^T \right\|^2_F \\
\\
& \stackrel{\rm{\eqref{eq:BxxB}}}{\leq}& mk + L \gamma \sum\limits_{j \in R}  \la A(E)x_j, x_j \ra^{-1} \sum\limits_{i=1}^m \left( \sum\limits_{l=1}^{nig} \la  B_{i, l}x_j, B_{i, l}x_j \ra \right)^2 \\
\\
& \leq &  mk + L \gamma \sum\limits_{j \in R}  \la A(E)x_j, x_j \ra^{-1} \left( \sum\limits_{i=1}^m \sum\limits_{l=1}^{nig} \la  B_{i, l}x_j, B_{i, l}x_j \ra \right)^2 \\
\\
& \stackrel{\rm{\eqref{eq:xAEx}}}{\leq} & mk + L \frac{\gamma}{r} \sum\limits_{j \in R} \sum\limits_{i=1}^m \sum\limits_{l=1}^{nig} \la  B_{i, l}x_j, B_{i, l}x_j \ra \\
\\
& \stackrel{\rm{\eqref{def:B}}}{\leq} & mk + L^2 \frac{\gamma}{r} \|B\|_2^2 \eta^2  .
\ea
\end{equation}
\end{proof}

Next, we give a bound on the norm of $g_{x}(E, x)$.

Let $\tilde{E}_i$ be the block diagonal matrix of $nig$ same diagonal blocks $E_i$.  Let $\tilde{E}$ be the block diagonal matrix with diagonal blocks $\tilde{E}_i, \, (i=1, \dots, m)$:
$$
\ba{rcl}
\tilde{E}_i \; \defeq \; \begin{bmatrix} E_i & & \\ & \ddots & \\ & & E_i \end{bmatrix} , \qquad
\tilde{E} \defeq \begin{bmatrix} \tilde{E}_1 & & \\ & \ddots & \\ & & \tilde{E}_m \end{bmatrix}.
\ea
$$
Then
$$
\ba{rcl}
A(E) & = & B^T \tilde{E} B .
\ea
$$
\begin{lemma}
The partial derivative of $F(E, x)$ in $x$ can be bounded as follows:
$$
\ba{rcl}
\left\| g_{x}(E, x) \right\|_{2,*} & \leq & 2\|f \|_2 + 2\sqrt{\gamma L (\rho_u - kr + r)} \|B\|_2 \; \defeq \; L_x.
\ea
$$
\end{lemma}
\begin{proof}
For a vector $z$ of proper dimension, we have
$$
\ba{rcl}
\|A(E)z\|_2 & \leq & \|A(E)\|^{1/2}_2 \la A(E)z, z \ra^{1/2} .
\ea
$$
And for two matrices $A$ and $B$ of proper dimension, it holds that $\lambda_{\max}(AB) \leq \lambda_{\max}(A) \lambda_{\max}(B)$, and $\lambda_{\max}(AB)= \lambda_{\max}(BA)$.

In addition, by the definition of $Q_k^{(i)}$, we have
$$
\ba{rcl}
\lambda_{\max}(\tilde{E}) \; = \; \lambda_{\max}(E) & \leq & \rho_u - (k-1)r .
\ea
$$

Therefore, $\|A(E)\|_2$ can be bounded as below:
$$
\ba{rcl}
\|A(E)\|_2^{1/2} & \leq & \|\tilde{E}\|_2^{1/2}\|B^TB\|_2^{1/2} \; \leq \;  \sqrt{\rho_u - (k-1)r} \|B\|_2 .
\ea
$$

Hence
\begin{equation} \label{bd:gx}
\ba{rcl}
\left\| g_{x}(E, x) \right\|_{2,*} & \leq & 2\|f \|_2 + 2\sqrt{\gamma} \Big( \sum_{j \in R} \la A(E) x_j, x_j \ra^{-1} \la A(E)x_j,  A(E) x_j \ra \\
\\
& + & \sum_{j \notin R} \la A(E)y, A(E)y \ra \Big)^{1/2} \\
\\
& \leq & 2\|f \|_2 + 2\sqrt{\gamma L (\rho_u - kr + r)} \|B\|_2 .
\ea
\end{equation}
\end{proof}

Next, we give bounds on the duality gaps.

By \cite[Lemma 3]{Nes07}, we have
\begin{equation} \label{bd:hatbeta}
\ba{rcl}
\sqrt{2t-1} & \leq & \hat{\beta}_t \; \leq \; \frac{1}{1+\sqrt{3}} + \sqrt{2t-1} , \quad t \geq 1 .
\ea
\end{equation}

\begin{theorem}
If the iterates are generated by the method of {\em Simple Dual Average}, the duality gap is bounded as
\begin{equation} \label{gap:mmxs}
\ba{rcl}
0 & \leq & \max\limits_{x_j \in Q_x} F(\hat{E}^{(t+1)}, x) - \min\limits_{E_i \in Q_k^{(i)}} F(E, \hat{x}^{(t+1)}) \\
\\
& \leq & \frac{0.37 + \sqrt{2t+1}}{t+1} \Big[
\sqrt{(mk + \frac{\gamma}{r} L^2 \|B\|_2^2\eta^2) m}(\rho_u - kr) \\
\\
& + & 2\left(\|f \|_2 + \sqrt{\gamma L (\rho_u - kr + r)} \|B\|_2\right)\sqrt{L} \eta
\Big] .
\ea
\end{equation}
\end{theorem}
\begin{proof}
Since partial subdifferentials of $f$ are uniformly bounded:
$$
\ba{c}
\|g_E\|_{F, *} \; \leq \; L_E, \quad \|g_x\|_{2, *} \; \leq \; L_x, \quad \forall \, (E, x) \in Q_k \times Q_x ,
\ea
$$
when we choose
$$
\ba{rcl}
\frac{1}{\tau} & = & 1 + \frac{L_x}{L_E}\sqrt{\frac{D_E}{D_x}}, \\
\\
\sigma & = & \sqrt{ \frac{\tau L^2_E + (1-\tau) L^2_x}{2 \tau D_E  + 2(1-\tau) D_x} } ,
\ea
$$
by \cite[(4.6)]{Nes07}, we have
\begin{equation} \label{gap:delta}
\ba{rcl}
\frac{1}{\sum_{l=0}^t \alpha_l }\delta_t & \leq & \frac{\hat{\beta}_{t+1}}{t+1} \sqrt{2}\left(L_E\sqrt{D_E} + L_x \sqrt{D_x} \right).
\ea
\end{equation}

Therefore, by \eqref{bd:gE}, \eqref{bd:gx}, \eqref{bd:gapt}, \eqref{bd:hatbeta}, \eqref{bd:DE}, \eqref{bd:SDADx}, we get
$$
\ba{rcl}
0 & \leq & \max\limits_{x_j \in Q_x} F(\hat{E}^{(t+1)}, x) - \min\limits_{E_i \in Q_k^{(i)}} F(E, \hat{x}^{(t+1)}) \\
\\
& \leq & \frac{0.37 + \sqrt{2t+1}}{t+1} \Big[
\sqrt{(mk + \frac{\gamma}{r} L^2 \|B\|_2^2\eta^2) m}(\rho_u - kr) \\
\\
& + & 2\left(\|f \|_2 + \sqrt{\gamma L (\rho_u - kr + r)} \|B\|_2\right)\sqrt{L} \eta
\Big] .
\ea
$$
\end{proof}

\begin{theorem}
If the iterates are generated by the method of {\em Weighted Dual Average}, the duality gap is bounded by
\begin{equation} \label{gap:mmxw}
\ba{rcl}
0 & \leq & \max\limits_{x_j \in Q_x} F(\hat{E}^{(t+1)}, x) - \min\limits_{E_i \in Q_k^{(i)}} F(E, \hat{x}^{(t+1)}) \\
\\
& \leq & \min \Bigg\{
\frac{0.37 + \sqrt{2t+1}}{t+1} \Big[ \sqrt{m^2k + L^2 m \frac{\gamma}{r}\|B\|_2^2 \eta^2}(\rho_u - kr) + 2 \sqrt{L} \eta \|f\|_2  \\
\\
& + & 2 \sqrt{\gamma(\rho_u - kr + r)}\|B\|_2 L \eta \Big], \; 
 \; 
\frac{(4\sqrt{2}+2)\hat{\beta}_{t+1} \sqrt{d(E^*, x^*)}}{t+1} \Big[mk + 8(3+\sqrt{2})\frac{\gamma}{r} L \|B\|_2^2 d(E^*, x^*) \\
\\
& + & 4\left(\|f\|_2 + \sqrt{\gamma L (\rho_u - kr + r)} \|B\|_2 \right)^2
\Big]^{1/2} \Bigg\} .
\ea
\end{equation}
\end{theorem}

\begin{proof}
\begin{enumerate}
\item{Bound 1}

Let $(E^*, x^*)$ be an optimal solution.
Since
\begin{equation} \label{def:dEx}
d(E, x) = \frac{\tau}{2} \|E-rI\|_F^2 + \frac{1-\tau}{2} \|x\|_2^2,
\end{equation}
we get
$$
\ba{rcl}
\sqrt{d(E, x)} & \leq & \sqrt{d(E^*, x^*)} + \frac{1}{\sqrt{2}} \left\| (E, x) - (E^*, x^*) \right\| .
\ea
$$
In addition, \cite[Theorem 3]{Nes07} states that
$$
\ba{rcl}
 \left\| (E, x) - (E^*, x^*) \right\|^2 & \leq & 2d(E^*, x^*) + \frac{1}{\sigma^2} .
\ea
$$

Therefore,
$$
\ba{rcl}
D_{E, x} & \defeq & \max\limits_{E_i \in Q_k^{(i)}, x \in Q_x} d(E, x) \; \leq \; \left( \sqrt{d(E^*, x^*)} + \sqrt{d(E^*, x^*) + \frac{1}{2\sigma^2}} \right)^2 \\
\\
& = & 2d(E^*, x^*) + \frac{1}{2\sigma^2} + 2 \sqrt{d(E^*, x^*)^2 + d(E^*, x^*)\frac{1}{2\sigma^2}} \\
\\
& \leq & 4d(E^*, x^*) + \frac{1}{2\sigma^2} + \sqrt{2d(E^*, x^*)}\frac{1}{\sigma} .
\ea
$$
By \cite[Theorem 3]{Nes07}, we further have
$$
\ba{rcl}
\delta_t & \leq & \hat{\beta}_{t+1}(D_{E, x} \sigma + \frac{1}{2\sigma} ) \; \leq \; \hat{\beta}_{t+1} \left[ 4d(E^*, x^*) \sigma + \frac{1}{\sigma} + \sqrt{2d(E^*, x^*)} \right].
\ea
$$

Minimizing the above last term in $\sigma$, we obtain that at $\sigma = 1/(2\sqrt{d(E^*, x^*})$,
\begin{eqnarray}
\delta_t & \leq & \hat{\beta}_{t+1}(4+\sqrt{2}) \sqrt{d(E^*, x^*)} , \label{bd:WDAdelta} \\
D_{E, x} & \leq & 2(3+\sqrt{2}) d(E^*, x^*) . \label{bd:WDAs}
\end{eqnarray}

Let $\tau = \frac{1}{2}$.  By \eqref{def:dEx}, \eqref{bd:gE} and \eqref{bd:WDAs}, we have
\begin{equation} \label{bd:SDAgE}
\ba{rcl}
L_E^2 \; = \; \max\limits_{E_i \in Q_k^{(i)}, x \in Q_x} \left\| g_{E}(E, x) \right\|_{F,*}^2 & \leq & mk + 8(3+\sqrt{2})\frac{\gamma}{r} L \|B\|_2^2 d(E^*, x^*) .
\ea
\end{equation}
Therefore,
$$
\ba{rcl}
\frac{1}{\sum_{l=1}^t \alpha_l} & \leq & \frac{1}{t+1} \sqrt{2L_E^2 + 2 L_x^2} \\
\\
& \stackrel{{\rm \eqref{bd:gx}, \eqref{bd:SDAgE}}}{\leq} &
\frac{1}{t+1} \Big[2mk + 16(3+\sqrt{2}) \frac{\gamma}{r} L \|B\|_2^2 d(E^*, x^*) \\
\\
& + & 2\left(2\|f \|_2 + 2\sqrt{\gamma L (\rho_u - kr + r)} \|B\|_2 \right)^2
\Big]^{1/2} .
\ea
$$

Along with \eqref{bd:gapt} and \eqref{bd:WDAdelta},  we obtain the duality gap
$$
\ba{rcl}
0 & \leq & \max\limits_{x_j \in Q_x} F(\hat{E}^{(t+1)}, x) - \min\limits_{E_i \in Q_k^{(i)}} F(E, \hat{x}^{(t+1)}) \\
\\
& \leq &
\frac{(4\sqrt{2}+2)\hat{\beta}_{t+1} \sqrt{d(E^*, x^*)}}{t+1} \Big[mk + 8(3+\sqrt{2})\frac{\gamma}{r} L \|B\|_2^2 d(E^*, x^*) \\
\\
& + & 4\left(\|f\|_2 + \sqrt{\gamma L (\rho_u - kr + r)} \|B\|_2 \right)^2
\Big]^{1/2}
\ea
$$

\item{Bound 2}

Since
$$
\ba{rcl}
\alpha_l & \geq & \frac{1}{ \sqrt{L_E^2/\tau + L_x^2/(1-\tau)} } ,
\ea
$$

by \cite[Theorem 3]{Nes07}, we have
$$
\ba{rcl}
0 & \leq & \max\limits_{x_j \in Q_x} F(\hat{E}^{(t+1)}, x) - \min\limits_{E_i \in Q_k^{(i)}} F(E, \hat{x}^{(t+1)}) \; \leq \; \frac{\delta_t}{\sum_{l=0}^t \alpha_l}\\
\\
& \leq &
\frac{\hat{\beta}_{t+1}}{t+1}\left[ \sigma(\tau D_E + (1-\tau)D_x) + \frac{1}{2 \sigma} \right] \sqrt{L_E^2/\tau + L_x^2/(1-\tau)}  .
\ea
$$

We choose
$$
\ba{rcl}
\sigma & = & \frac{1}{\sqrt{2\tau D_E + 2(1-\tau)D_x}}, \quad
\tau \; = \; \frac{\sqrt{D_x} L_E}{\sqrt{D_E} L_x + \sqrt{D_x} L_E} .
\ea
$$
Then
$$
\ba{rcl}
0 & \leq & \max\limits_{x_j \in Q_x} F(\hat{E}^{(t+1)}, x) - \min\limits_{E_i \in Q_k^{(i)}} F(E, \hat{x}^{(t+1)}) \\
\\
& \leq &  \frac{\sqrt{2}\hat{\beta}_{t+1}}{t+1}\left( L_E \sqrt{D_E} + L_x \sqrt{D_x} \right) .
\ea
$$

From \eqref{bd:hatbeta}, \eqref{bd:gE}, \eqref{bd:gx}, \eqref{bd:DE}, and \eqref{bd:SDADx}, we have
$$
\ba{rcl}
0 & \leq & \max\limits_{x_j \in Q_x} F(\hat{E}^{(t+1)}, x) - \min\limits_{E_i \in Q_k^{(i)}} F(E, \hat{x}^{(t+1)}) \\
\\
& \leq & \frac{0.37 + \sqrt{2t+1}}{t+1} \Big[ \sqrt{m^2k + L^2 m \frac{\gamma}{r}\|B\|_2^2 \eta^2}(\rho_u - kr) + 2 \sqrt{L} \eta \|f\|_2  \\
\\
& + & 2 \sqrt{\gamma(\rho_u - kr + r)}\|B\|_2 L \eta \Big] .
\ea
$$

\end{enumerate}
\end{proof}

\subsection{Computational Cost of Each Iteration} \label{SS:ComputationCost}

The costs of each iteration of our algorithm have two components: that from calculating the subgradients and that from solving the subproblems.

\begin{enumerate}
\item \textit{Cost of updating $s^E$ and $s^x$}.
$$
\ba{rcl}
0 & \leq & \max\limits_{x_j \in Q_x} F(\hat{E}^{(t+1)}, x) - \min\limits_{E_i \in Q_k^{(i)}} F(E, \hat{x}^{(t+1)}) \\
\\
& \leq &
\frac{(4\sqrt{2}+2)\hat{\beta}_{t+1} \sqrt{d(E^*, x^*)}}{t+1} \Big[mk + 8(3+\sqrt{2})\frac{\gamma}{r} L \|B\|_2^2 d(E^*, x^*) \\
\\
& + & 4\left(\|f\|_2 + \sqrt{\gamma L (\rho_u - kr + r)} \|B\|_2 \right)^2
\Big]^{1/2}
\ea
$$

We don't keep $g_E$ and $g_x$ in memory, but update $s^E_{t+1}$ and $s^x_{t+1}$ directly.
Since $g_E$ and $g_x$ share some same components, we compute $\sum_{l=1}^{nig} B_{i,l} x_j x_j^T B_{i,l}$ and $A(E)x_j$ in the same loop.
To balance the demands between memory and speed, we compute $s^E_{t+1}$ and $s^x_{t+1}$ as follows:

\makebox[3.9in]{\hrulefill}
\begin{alltt}
do j = 1 \dots L
  if \( j \in R \)
   \( 0 \rightarrow w \)
   \( \alpha\sb{t}\sqrt{\gamma}\la A(E)x\sb{j}, x\sb{j} \ra\sp{-1/2}\rightarrow u\sb{j} \)
    do i = 1 \dots m
     \( \sum\limits\sb{l=1}\sp{nig} B\sb{i,l}x\sb{j}x\sb{j}\sp{T}B\sb{i,l}\sp{T} \rightarrow q \)
     \( w + A\sb{i}(E)x\sb{j} \rightarrow w \)
     \( s\sp{E\sb{i}} - u\sb{j}q \rightarrow s\sp{E\sb{i}} \)        (\(k(k+1)\) flops)
    end i
  else
   \( A(E)y \rightarrow w \)             (\(m{\cdot}nig(4kN+2k\sp{2})\) flops)
  end if
 \( s\sp{x\sb{j}} + 2(\alpha\sb{t}f\sb{j} - u\sb{j}w) \rightarrow s\sp{x\sb{j}} \) (\(5N\) flops)                           
end do j
\(s\sp{E} + \alpha\sb{t}I\sb{k} \rightarrow s\sp{E}\)               (\(mk\) flops)
\end{alltt}
\makebox[3.9in]{\hrulefill}\\

\vspace{12pt}

The inner products $\la A(E) x_j, x_j \ra$ are computed as follows:

\makebox[3.5in]{\hrulefill}
\begin{alltt}
\(0 \rightarrow u \)
do i = 1 \dots m 
  do l = 1 \dots nig
    \( B\sb{i,l}x\sb{j} \rightarrow v \)     (\(2kN\) flops)
    \( E\sb{i}v \rightarrow p \)       (\(2k\sp{2}\) flops)
    \( u+v\sp{T}p \rightarrow u \)   (\(2k\) flops)
  end do l
end do i
output s 
\end{alltt}
\makebox[3.5in]{\hrulefill}

In the algorithm, we keep the value $\sqrt{\gamma}$ in memory.
Therefore, the arithmetic costs of calculating $\alpha_t \sqrt{\gamma}\la A(E)x_j, x_j \ra^{-1/2}$ for $j=1, \dots, L$ are $\big( 2L \cdot m \cdot nig \cdot [kN + k^2 + k] + 4 L \big)$ flops.
The total length of auxiliary vectors $v, p$ and $u_j$ is $(N+k+L)$.
After computing the $u_j$'s, memory for $v$ and $p$ can be released.

We compute $(\sum_{l=1}^{nig} B_{i,l} x_j x_j^T B_{i,l}^T)$ for $g_E$ and $A(E)x_j$ for $g_x$ in the same loop; i.e., the update of $q$ and $w$ in loop $i$ of the above algorithm is done as follows:

\makebox[3.5in]{\hrulefill}
\begin{alltt}
\(q \rightarrow 0\)
do l = 1 \dots nig
 \( B\sb{i,l}x\sb{j} \rightarrow v \)               (\(2kN\) flops)
 \( q + vv\sp{T} \rightarrow q \)            (\(k(k+1)\) flops)
 \( E\sb{i}v \rightarrow v \)                 (\(2k\sp{2}\) flops)
 \( w + B\sb{i,l}\sp{T}v \rightarrow w \)          (\(2kN\) flops)
end do l
\end{alltt}
\makebox[3.5in]{\hrulefill}

The above \texttt{l} loop takes a total of $nig(4kN + 3k^2 + k)$ flops.  And it is executed at most $Lm$ times.  The total length of the auxiliary vectors $v, q$ and $w$ is $(k+\frac{k(k+1)}{2}+N)$. 

Adding all together, we get that the total number of flops used in updating $s^E$ and $s^x$ is at most $(6kL\cdot nig) mN + [(5k^2+3k)L \cdot nig + (k^2+k)L+ k]m + (5L)N + 4L$.  And at most $\big(\frac{1}{2} k^2 +\frac{3}{2}k + N +L\big)$ auxiliary storage space units are used.

\item \textit{Cost of solving the subproblems.}

For $t=0, \dots$ , from the closed-from solution \eqref{eq:solutionx} given in \S~\ref{S:subsolution}, we obtain that it takes $L(3N+7)$ flops to compute $x^{(t+1)}$.  The value of $\beta_{t+1} \tau$ is stored for calculating $E^{(t+1)}$ later.

Now we consider the worst-case complexity of computing $E^{(t+1)}$.
By the representation of $E^{(t+1)}$, it is obvious that the most computation is needed when
$$
\lambda < 0 ,  \qquad  \sum_{i=1}^k \lambda_i < \beta_{t+1} \tau(kr - \rho_u) .
$$ 
Comparing $\sum_{q \in \bar{M}_0} \lambda_q$ with $\beta_{t+1} \tau (kr - \rho_u)$ and $\beta_{t+1} \tau (kr - \rho_l)$ takes $(2k+7)$ flops and $3$ auxiliary storage space units, since we can keep $kr$ as an intermediate result.  Similarly to the analysis in \S\S\ref{SS:MProj} of \textit{Appendix: Matrix Projection}, we can obtain the complexity of \textit{Algorithm projSyml} as follows:  \textit{Step 1} takes at most $k(k-1)$ comparisons and exchanges. Because we have already calculated $\beta_{t+1} \tau(kr-\rho_u)$, $2$ additions and subtractions are needed to obtain $T$.  \textit{Step 2} takes at most $3(k-1)$ flops.  \textit{Step 3} takes at most $(2 + 3k)$ steps.
Therefore, a total of at most $(k^2 + 7k+8)$ flops are needed to obtain $\omega$.  And $(k + 4)$ auxiliary space units are needed to store the sorted index set, $T$, $\beta_{t+1} \tau$, $\beta_{t+1} \tau(kr - \rho_u)$, $q$, since we overwrite the memory storing $\beta_{t+1}\tau(kr - \rho_l)$ by $T$.

Eigenvalue decomposition of $s_{t+1}^{E_i}$ takes about $9k^3$ flops and $k^2+2k+1$ auxiliary storage space units.  Computing $U \diag(\omega) U^T$ takes about $(k^2(k+1)+k^2)$ flops.
Therefore, at most $m(10 k^3 + 3k^2 + 7k +8)$ flops and $(k^2 + 2k + 1)$ auxiliary storage space units are needed to obtain $E^{(t+1)}$.

\end{enumerate}

For problem \eqref{MLSDP}, $k$ equals to $3$ or $6$; $L$ and $nig$ are much smaller than $m$ or $N$.  After omitting small-order terms, we then conclude that about $(6kL\cdot nig)mN$ flops are needed for each iteration of our algorithm.  And the auxiliary storage space units are about $N$.

On the other hand, to evaluate $\la A(E)^{-1} f_j, f_j \ra$ presented in the original formula \eqref{MLSDP}, we need to first form the matrix $A(E)$, which requires $m \cdot nig \cdot[2k^2N + (k + \frac{1}{2})N(N+1)]$ flops: computing $E_iB_{i,l}$ takes $2k^2N$ flops; calculating $B_{i,l}^T E_i B_{i,l}$ for $kN(N+1)$ flops; adding the $m \cdot nig$ matrices $B_{i,l}^T E_i B_{i,l}$ together requires $\frac{1}{2} m \cdot nig N(N+1)$ flops. 
An auxiliary vector of size $\frac{N(N+1)}{2}$ is needed to store $A(E)$.
We then compute the Cholesky factorization of $A(E) = CC^T$, which takes $\frac{N^3}{3}$ flops.  Next we compute $z_j = C^{-T}(C^{-1} x_j)$ (for $j=1, \dots, L$), which needs $2LN^2$ flops.  Finally, the inner products $\la z_j, x_j \ra$ takes $2LN$ flops to compute.
Therefore, a total of $m \cdot nig \cdot[2k^2N + (k + \frac{1}{2})N(N+1)] + \frac{N^3}{3} + 2L(N^2+N)$ flops and an auxiliary vector of size $\frac{N(N+1)}{2}$ are required to compute $\la A(E)^{-1} f_j, f_j \ra$ (for $j=1, \dots, L$). 
After omitting small-order terms, we conclude that about $\frac{1}{3}N^3$ flops and $\frac{1}{2} N^2$ auxiliary storage space units are needed to obtain $\la A(E)^{-1} f_j, f_j \ra$.

In summary, the number of flops and auxiliary storage space units per iteration of our algorithm are both one order smaller than that for evaluating $\la A(E)^{-1} f_j, f_j \ra$.
Furthermore, if the matrices $B_{i,l}$ are sparse, computational work per iteration and auxiliary storage space requirement of our algorithm will be even smaller.

\section{Penalized Lagrangian} \label{S:penalty}
Because $\la A(E)^{-1} f_j, f_j \ra$ is convex in $E$, and the function $( [\sqrt{a}-\gamma^{1/2}]_+ )^2$ is convex and increasing in $a$, we conclude that
$\big([\la A(E)^{-1} f_j, f_j \ra^{1/2} - \gamma^{1/2}]_+ \big)^2$ is convex in $E$; see, for instance~\cite[Proposition 2.1.8]{MR1261420}. 
To have a faster rate of convergence to feasibility, we add to the objective of \eqref{MLSDP} a convex penalty function for the compliance constraint:
$$
\ba{c}
\sum\limits_{j=1}^L \nu \left( \left[ \la A(E)^{-1}f_j, f_j \ra^{1/2} - \gamma^{1/2} \right]_+ \right)^2 ,
\ea
$$
where $\nu > 0$ is the penalty parameter.

Then the Lagrangian becomes
$$ \ba{rcl}
p(E, x) & \defeq & F(E, x) + 
\sum\limits_{j=1}^L \nu \left( \left[ \la A(E)^{-1}f_j, f_j \ra^{1/2} - \gamma^{1/2} \right]_+  \right)^2,
\ea
$$
which is convex in $E$ and concave in $x$.
And a solution to 
$$
\ba{c}
\min\limits_{\stackrel{E_i \in Q_k^{(i)}}{k=1, \dots, m}} \max\limits_{\stackrel{x_j \in \mathbb{R}^N,}{j=1, \dots, L}} p(E, x)
\ea
$$
approximate that of model~\eqref{MLSDP}.

The gradient of $p(E, x)$ at $(E, x)$ is
$$
\ba{rcl}
\nabla_{E_i} p(E,x) & = & g_{E_i}(E,x)  -  \sum\limits_{j \in W_E} \nu \left[ 1 - \gamma^{1/2} / \la A(E)^{-1}f_j, f_j \ra^{1/2} \right]_+  
\\
\\
& & \cdot \left( \sum\limits_{l=1}^{nig} B_{i,l} A(E)^{-1} f_j f_j^T A(E)^{-1} B_{i,l}^T \right)  \quad i = 1, \dots, m ,
\\
\\
\nabla_{x_j} p(E,x) & = & g_{x_j} (E, x)  \quad j = 1, \dots, L .
\ea
$$

Similar to Lemma~\ref{lm:bdL}, we have the following results about the bounded version of penalized Lagrangian method.
\begin{lemma}
Let $(\tilde{E}, \tilde{x})$ be a solution to 
\begin{equation} \label{PBSPF}
\min\limits_{\stackrel{E_i \in Q_k^{(i)}}{i=1, \dots, m}} \max\limits_{\stackrel{\|x_j\| \leq \eta,}{j=1, \dots, L} } p(E, x) .
\end{equation}

Let $f^*$ be the optimal value of \eqref{MLSDP}.
\begin{enumerate}
\item \label{item:bEInt}
If $\|\tilde{x}_j\| < \eta$ for $j=1, \dots, L$; then $\tilde{E}$ is a solution to the original problem.
\item \label{item:bEBd}
Otherwise, $\tilde{E}$ has the following properties:
\begin{enumerate}
\item
$ F(\tilde{E}) \leq  F^*$.
\item
$\sum\limits_{j \in W_{\tilde{E}}} \left( \la f_j, A(\tilde{E})^{-1} f_j \ra^{1/2} - \gamma^{1/2} \right) 
\leq$ \\
$1/\left[
\sqrt{\frac{\nu}{(f^*-m\rho_l)|W_{\hat{E}}|} + \frac{ r^2\lambda^2_{\min}(BB^T)\eta^2 }{ (f^* - m\rho_l)^2 } }
+ \frac{ r\lambda_{\min}(BB^T)\eta }{ f^* - m\rho_l } 
\right] $ .
\end{enumerate}
\end{enumerate}
\end{lemma}

\begin{proof}
Proof for Item~\ref{item:bEInt} is the same as that for Lemma~\ref{lm:bdL}.
Item~\ref{item:bEBd} can be proved similarly as Lemma~\ref{lm:bdL}.
Below, we briefly give the proof.

For any fixed $E \in Q$, the point
$$
\ba{rcl}
x_j & = & \begin{cases} \frac{\eta}{\|A(E)^{-1}f_j\|}A(E)^{-1}f_j & j \in W_E \\
0 & j \notin W_E \end{cases}
\ea
$$
is feasible to
$$
\max\limits_{\stackrel{\|x_j\| \leq \eta,}{ j=1, \dots, L} } p(E, x) 
$$
with objective value
$$
\ba{rcl}
p_x(E)  \defeq  \la I, E \ra + 
\sum\limits_{j \in W_E} \nu w_j^2(E) +2 \frac{\la f_j, A(E)^{-1} f_j \ra^{1/2} }{\|A(E)^{-1}f_j\|} \eta w_j(E) ,
\ea
$$
where
$$
\ba{rcl}
w_j(E) & \defeq & 
\la f_j, A(\tilde{E})^{-1} f_j \ra^{1/2} - \gamma^{1/2} , \quad
j \in W_E .
\ea
$$
Therefore,
$$
\ba{rcl}
F^* -  m \rho_l & \geq & \sum\limits_{j \in W_{\tilde{E}}} \nu w_j^2(\hat{E}) + 2r\lambda_{\min}(BB^T) \eta w_j(\hat{E}) ,
\ea
$$
from which we obtain
$$
\ba{c}
F^* -  m \rho_l + |W_{\hat{E}}|  r^2\lambda^2_{\min}(BB^T) \eta^2/\nu  
\; \geq \; \sum\limits_{j \in W_{\tilde{E}}} \nu \left[ w_j(\hat{E}) + r\lambda_{\min}(BB^T)\eta / \nu \right]^2 \\
\\
\; \geq \; \frac{\nu}{|W_{\hat{E}}|} \left[ \sum\limits_{j \in W_{\tilde{E}}}  w_j(\hat{E}) + |W_{\hat{E}}| r\lambda_{\min}(BB^T)\eta / \nu \right]^2 .
\ea
$$
Hence,
$$
\ba{rcl}
\sum\limits_{j \in W_{\tilde{E}}}  w_j(\hat{E}) & \leq &
1/\left[
\sqrt{\frac{\nu}{(f^*-m\rho_l)|W_{\hat{E}}|} + \frac{ r^2\lambda^2_{\min}(BB^T)\eta^2 }{ (f^* - m\rho_l)^2 } }
+ \frac{ r\lambda_{\min}(BB^T)\eta }{ F^* - m\rho_l } 
\right] .
\ea
$$

\end{proof}
Observe that as $\eta \rightarrow + \infty$ and $\nu \rightarrow + \infty$, the set of saddle-points of \eqref{PBSPF} approaches that of \eqref{MLSDP}.

We can apply the preceding algorithm to obtain a saddle-point of $p(E, x)$ as well.  And its subproblems have closed-form solutions.

\paragraph{Bounds on duality gaps}
To estimate the duality gap of each iteration, We first bound $\nabla_E p(E,x)$ as follows
$$
\ba{rcl}
\left\| \nabla_E p(E,x) \right\|_{F, *} & \leq & L_E + \nu \left[
\sum\limits_{i=1}^m  \tr \left( \sum\limits_{j=1}^L
\sum\limits_{l=1}^{nig} B_{i,l} A(E)^{-1} f_j f_j^T A(E)^{-1} B_{i,l}^T \right)^2
\right]^{1/2} \\
\\
& \leq & L_E +
\nu \tr \left( \sum\limits_{i=1}^m \sum\limits_{j=1}^L 
\sum\limits_{l=1}^{nig} B_{i,l} A(E)^{-1} f_j f_j^T A(E)^{-1} B_{i,l}^T \right)
\\
\\
& = & L_E +
\nu \sum\limits_{j=1}^L \sum\limits_{i=1}^m 
\sum\limits_{l=1}^{nig} 
f_j^T A(E)^{-1} B_{i,l}^T  B_{i,l} A(E)^{-1} f_j  \\
\\
& \stackrel{\lambda_{\min}(E) = r}{\leq} & L_E + \nu /r \sum\limits_{j=1}^L f_j^T A(E)^{-1} f_j  \\
\\
&\leq&
L_E + \frac{\nu}{r^2\lambda_{\min}(B^TB)} 
\sum\limits_{j=1}^L \|f_j\|_2^2 ,
\ea
$$
where the last inequality is from
$$
\lambda_{\min} \left( A(E) \right) \geq \lambda_{\min}(E) \lambda_{\min}(B^TB)\; = \; r\lambda_{\min}(B^TB) .
$$

By \eqref{bd:DE}, \eqref{gap:delta}, \eqref{gap:mmxs}, and \eqref{gap:mmxw} in \S\ref{S:boundd}, we obtain that the duality gaps of the iterates are bounded as follows: for $t = 0, \dots$,
$$
\ba{rcl}
0 \; \leq \; \max\limits_{x_j \in Q_x} F(\hat{E}^{(t+1)}, x) - \min\limits_{E_i \in Q_k^{(i)}} F(E, \hat{x}^{(t+1)})
\; \leq \; gap +
\frac{0.37 + \sqrt{2t+1}}{t+1} \frac{\sqrt{m}(\rho_u - kr)\nu}{r^2 \lambda_{\min}(B^TB)} \sum\limits_{j=1}^L \|f_j\|_2^2 . 
\ea
$$

\paragraph{Cost of each iteration.}
Compared with \eqref{BSPF}, extra computation is needed to calculate $\la A(E)^{-1} f_j, f_j \ra$ for $j=1, \dots, L$ in order to solve \eqref{PBSPF}, which is $\mathcal{O}(N^3)$ flops; see the analysis at the end of \S\ref{S:boundd}.  Therefore, the total cost of each iteration for solving \eqref{PBSPF} is $\mathcal{O}(N^3)$ flops and $\mathcal{O}(N^2)$ memory space units.

\section{Numerical Examples} \label{S:computing}
Below, we present some computational examples which are done in the MATLAB environment on a windows PC.
For each run, the starting point is as follows:  We choose $E_0$ to be the identity matrix with trace equals to the upper bound of trace.  For $j=1, \dots, L$, we let $x_j$ be a vector with the same element and  $ \|x_j\|_2 = \eta$. 

Figure~\ref{fig:tc18s1} shows how the objective value and the violation of constraints vary with the number of iterations.
The problem instance is \texttt{tc18\_s1} from the academic test library of the Plato project (\url{www.plato-n.org}) with $m=128$, $N=298$, $L=1$, $nig=4$. 

\begin{figure}[!ht] 
\begin{center}
\caption{An Example From the Academic Test Library} \label{fig:tc18s1}
\includegraphics[scale=.6]{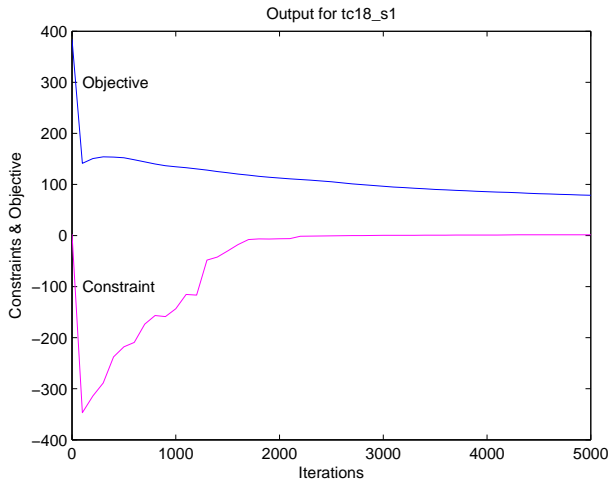}
\end{center}
\end{figure}
The figure shows that during the first few iterations the objective value decreases but the constraint violation increases rapidly,
where constraint violation is measured by $\sum_{j=1}^L \min \big[ (\la A(E)^{-1} f_j, f_j \ra - \gamma), 0 \big]$.
With iterations moving on, the constraint violation decreases with the objective value.
\FloatBarrier
In Table~\ref{Tb:atl} and Table~\ref{Tb:pl}, we present further numerical results on problems in the academic test library of the Plato project (\url{www.plato-n.org}).
In the tables, columns `cpu' give the total CPU times in seconds; columns `obj' give the final objective values; column `obj0' gives the initial objective values; column `const' indicates whether the constraints are satisfied or not for the final solutions: `f' means feasible.
We compare formulas \eqref{BSPF} and \eqref{PBSPF} on some infeasible problems, because constraints of these problems are difficult.
The results are presented in Table~\ref{Tb:pl}.
For each instance, we run $5000$ gradient iterations.
In Table~\ref{Tb:pl}, columns `const' give the sum of the values of the violation of constraints; i.e. $\sum_{j=1}^L \min \big[ (\la A(E)^{-1} f_j, f_j \ra - \gamma), 0 \big]$.  Columns `obj-0' and `const-0' give objective values and the sum of values of the violation of constraints for the initial solutions.  Columns `obj-p' and `const-p'  give objective values and the sum of the values of the violation of constraints of the final solutions obtained by model~\eqref{PBSPF}.  Columns `obj' and `const' give objective values and the sum of the values of the violation of constraints of the final solutions obtained by model~\eqref{BSPF}.

\begin{table}[htbp]
\begin{center}
\caption{Examples on Problems in Academic Test Library} \label{Tb:atl}
\begin{tabular} {|c|c|c|c|c|c||c|c|c|} \hline
\multicolumn{6}{|c||}{problem} & \multicolumn{3}{c|}{gradient method} \\
\hline
prob& m & N & L & nig & obj0  & cpu & obj & const \\ \hline
tc01\_s1 & 96 & 216 & 1 & 4 & 288  & 2.77e+2 & 61.21 & f \\
tc01\_s2 & 384 & 816 & 1 & 4 & 1152 & 1.73e+3 & 8.82e+2 & f \\
tc02\_s1 & 96 & 216 & 1 & 4 & 288 & 2.96e+2 & 4.29 & f \\
tc02\_s2 & 384 & 816 & 1 & 4 & 1152 & 1.85e+3 & 6.43 & f \\
tc03\_s1 & 96 & 216 & 1 & 4 & 288  & 2.77e+2 & 60.12 & f \\
tc03\_s2 & 384 & 816 & 1 & 4 & 1152 & 1.68e+3 & 421.08 & f \\
tc04\_s1 & 300 & 670 & 1 & 4 & 900 & 1.26e+3 & 546.79 & f \\
tc05\_s1 & 800 & 1719 & 1 & 4 & 2.4e+3 & 5.19e+3 & 1.54e+3 & f \\
tc07\_s1 & 800 & 1680 & 1 & 4 & 2.4e+3  & 5.11e+3 & 6.75e+2 & f \\
tc08\_s1 & 128 & 272 & 1 & 4 & 384 & 3.79e+2 & 17.42 & f \\
tc08\_s2 & 512 & 1056 & 1 & 4 & 1536 & 2.48e+3 & 8.78e+2 & f \\
tc14\_s1 & 100 & 248 & 1 & 4 & 300 &  3.2e+2 & 51.98 & f \\
tc14\_s2 & 400 & 898 & 1 & 4 & 1200  &  1.8e+3 & 161.46 & f \\
tc16\_s1 & 128 & 300 & 1 & 4 & 384  &  3.9e+2 & 50.73 & f \\
tc16\_s2 & 512 & 1116 & 1 & 4 & 1536 &  2.74e+3 & 973.14 & f \\
tc17\_s1 & 128 & 300 & 1 & 4  & 384 & 4.14e+2 & 1.66e+2 & f \\
tc17\_s2 & 512 & 1116 & 1 & 4 & 1536  & 2.688e+3 & 5.54e+2 & f \\
tc18\_s1 & 128 & 298 & 1 & 4 & 384  & 4.0e+2 & 74.81 & f \\
tc18\_s2 & 512 & 1114 & 1 & 4 & 1536  & 2.57e+3 & 418.78 & f \\
tc18sl\_s1 & 128 & 298 & 1 & 4 & 384 &  3.85e+2 & 0.79 & f \\
tc18sl\_s2 & 512 & 1114 & 1 & 4 & 1536 & 2.57e+3 & 62.76 & f \\
tc03\_s1 & 96 & 216 & 2 & 4 & 288  & 4.24e+2 & 65.01 & f \\
tc03\_s2 & 384 & 816 & 2 & 4 & 1152  & 3.19e+3 & 739.23 & f \\
tc06\_s1 & 800 & 1719 & 3 & 4 & 2.4e+3  & 1.37e+4 & 1.5e+3 & f \\
tc16\_s1 & 128 & 300 & 2 & 4 & 384  &  5.91e+2 & 183.30 & f \\
tc16\_s2 & 512 & 1116 & 2 & 4 & 1536  &  4.32e+3 & 298.33 & f \\
tc17\_s1 & 128 & 300 & 2 & 4 & 384  & 5.99e+2 & 211.52 & f \\
tc17\_s2 & 512 & 1116 & 2 & 4 & 1536 & 4.04e+3  & 566.61  & f \\
tc09\_s1 (3d) & 100 & 567 & 4 &  8 & 300 & 2.56e+3 & 73.95 & f \\
tc09\_s2 (3d) & 512 & 2250 & 4 & 8 & 1536 & 3.66e+4 & 417.6 & f \\
tc10\_s1 (3d) & 100 & 567  & 2  & 8 & 300  & 1.62e+3 & 51.71 & f \\
\hline
\end{tabular}
\end{center}
\end{table}

\begin{sidewaystable}[htbp]
\begin{center}
\caption{Results With and Without Penalty Function} \label{Tb:pl}
\begin{tabular} {|c|c|c|c|c|c|c||c|c|c||c|c|c|} \hline
\multicolumn{7}{|c||}{problem} & \multicolumn{3}{c||}{with penalty} & \multicolumn{3}{c|}{without penalty} \\
\hline
prob & m & N & L & nig & obj-0 & const-0 & cpu-p & obj-p & const-p & cpu & obj & const \\ \hline
bmat2x2 (1) & 4 & 114 & 2 & 4 & 20 & -10.99e+3 & 22.38 & 20 & -6.87e+3 & 15.89 & 10.09 & -1.7e+6 \\
bmat2x2 (2) & 4 & 114 & 2 & 4 & 12 & -6.77e+6 & 22.55 & 12 & -4.4e+6 & 15.25 & 5.53 & -2.52e+9 \\
bmat2x2 (3) & 4 & 114 & 2 & 4 & 12 & -8.75e+4 & 24.17 & 12  & -1.02e+5 & 17.64 & 11.78 & -8.03e+5 \\
bmat1 (1) & 16 & 40 & 1 & 4 & 80 & -1.21e+4 & 63 & 79.93 & -4.5e+3 & 41.89 & 20.09 & -1.14e+7 \\
bmat1 (2) & 16 & 40 & 2 & 4 & 80 & -3.17e+3 & 88.89 & 79.99 & -3.39e+3 & 59.73 & 25.16 & -1.25e+5\\
bmat1 (3) & 16 & 40 & 2 & 4 & 48 & -3.3e+6 & 86.58 & 47.99 & -3.04e+6 & 59.66 & 36.29 & -4.48e+6 \\
bmat2 (1) & 200 & 440 & 1 & 4 & 600 & -1.98e+1 & 2.32e+4 & 573.07 & -6.49 & 661.6 & 185.74 & -6.44+3
\\
bmat2 (2) & 200 & 440 & 2 & 4 & 600 & -3.02e+2 & 2.33e+4 & 303.04 & -2.18e+2 & 1.03e+3 & 103.26 & -2.82e+4
\\
bmat2 (3) & 200 & 440 & 2 & 4 & 1.0e+3 & -7.3e+4 & 2.45e+4 & 785.73 & -5.52e+4 & 1.03e+3 & 188.49 &-1.49e+7
\\
bmat (1) & 400 & 850 & 1 & 4 & 2.0e+3 & -9.11 & 1.64e+5 & 1991.12 & -7.17 & 1.77e+3 & 794.46 & -6.07e+3 \\
bmat (2) & 400 & 850 & 1 & 4 & 1.2e+3 & -2.94e+2 & 1.66e+5 & 1175.92 & -1.84e+2 & 2.79e+3 & 1121.55 & -3.8e+2 \\
bmat (3) & 400 & 850 & 1 & 4 & 1.2e+3 & -2.72e+2 & 1.66e+5 & 1187.93 & -1.72e+2 & 3.14e+3 & 218.68 & -4.18e+2 \\
bmat1g & 400 & 850 & 1 & 4 & 600 & -1.22e+5 & 2.35e+4 & 488.4 & -9.28e+4 & 1.07e+3 & 118.96 & -5.78e+6 \\
\hline
\end{tabular}
\end{center}
\end{sidewaystable}

From the results in Table~\ref{Tb:pl}, we see that the penalized Lagrangian can produce a better solution for infeasible problems, although it may not be the case for feasible problems.  The penalty term forces iterates to move to the feasible region.  On the other hand, because each iteration is much cheaper without calculating $\la A(E)^{-1} f_j, f_j \ra$, the penalized Lagrangian takes longer to solve a problem instance of FMO.
The larger the dimension of the problem, the less time model~\eqref{BSPF} used compared with model~\eqref{PBSPF}.

\setcounter{section}{1}
\alph{subsection}
\section*{Acknowledgement}
We thank the associate editor and anonymous referees for their helpful comments and suggestions.

\section*{Appendix: Matrix Projection} \label{S:mproj}
Let $\Herm^n$ denote the space of $n \times n$ Hermitian matrices.
We take the standard inner product on the space of complex square matrices of order $n$ (or linear operators between Hilbert spaces of same dimension): $\forall \, U, V \in \mathbb{C}^{n\times n}$,
$$
\ba{rcl}
\la U, V \ra & = & \tr (UV^*) ,
\ea
$$
where $V^*$ is the conjugate transpose of $V$. Let $\|\cdot \|_F$
denote the corresponding Frobenius norm.  In this part, we give a closed-form solution to the following projection problem:
\begin{equation} \label{eq:projH}
\ba{rc}
\min\limits_{Z \in \Herm^n} & \|Z - U \|_F \\
\ST & c_l \leq \tr(Z) \leq c_u \\
& \lambda_{\min}(Z) \geq r ,
\ea
\end{equation}
where $U$ is a square complex matrix of order $n$.

To this end, we first consider a least squares problem with nonnegativity constraint and a two sided inequality.

\subsection{Least squares with a two-sided inequality and non-negative variables}
Least squares problems have been studied intensively; however, we cannot find any reference for the problem discussed in this section elsewhere.
In this part, we first give an analytical solution of the problem; then we present an algorithm with total number of operations being a quadratic term in the dimension of problem variable.

Given $A \in \mathbb{R}^{n \times n}$ be diagonal, $b \in \mathbb{R}^n$, $w \in \mathbb{R}^n$, $r \in \mathbb{R}^n$, $c_l \in \mathbb{R} \cup \{-\infty\}$, $c_u \in \mathbb{R} \cup \{+\infty\} $ with $c_l \leq c_u$.
Let$\|\cdot \|_2$ denote the norm induced by the inner product $\la \cdot, \cdot \ra$.
In this part, we give an analytical solution for the following least squares problem:
\begin{equation} \label{eq:GLS}
\ba{rcl}
\min\limits_{z \in \mathbb{R}^n}& \left\| Az - b \right\|_2^2 \\
\\
\ST & c_l \; \leq \; \la w, z \ra \; \leq \; c_u \\
\\
& z \; \geq \; r .
\ea
\end{equation}
Note that our problem includes the one-side inequality case when $c_l=-\infty$ or $c_u = + \infty$, the lower bounded variable case when $c_l=-\infty$ and $c_u = + \infty$, one equality case when $c_l = c_u$.
Our problem also includes the case when not all variables are bounded, since we can replace an unconstrained variable $z_i \in \mathbb{R}$ by $z_i = z_i^+ - z_i^-$ with $z_i^+ \geq 0$,   $z_i^- \geq 0$.

\subsubsection{Problem Reduction}
To solve problem~\eqref{eq:GLS}, we first show that we only need to consider the case with $A$ being identity and $w_i \neq 0$ for $i=1, \dots, n$.

If there exists ($\exists$) $ a_{ii} =0, w_i =0$, we let
$$
z_i^* \; = \; r_i .
$$

If $\exists \, a_{ii} =0, w_i >0$, we let
$$
z_i^* \; = \; \max\left\{
\left[ c_l - \sum\limits_{(1 \leq j \leq n \colon a_{jj} \neq 0} w_j \max\left\{ b_j/a_{jj}, r_j \right\}
\right]_+/w_i, r_i
\right\}.
$$

If $\exists \, a_{ii} =0, w_i < 0$, we let
$$
z_i^* \; = \; \max\left\{
-\left[ \sum\limits_{(1 \leq j \leq n \colon a_{jj} \neq 0)} w_j \max\left\{ b_j/a_{jj}, r_j \right\}
-c_u \right]_+/w_i, r_i
\right\}.
$$

We also replace $c_u$ and $c_l$ by 
$$
c_u - w_i z_i^*, \quad c_l - w_i z_i^* .
$$

If $\exists \, a_{ii} < 0$, we replace $a_{ii}$ with $-a_{ii}$ and $b_i$ with $-b_i$.

Hence after simplification, we can assume that $A$ is a positive diagonal matrix in the text below.  And our least squares problem is equivalent to
$$
\ba{rcl}
\min\limits_{z \in \mathbb{R}^n}& \left\| z - b + Ar \right\|_2^2 \\
\\
\ST & c_l - \la w, r \ra \; \leq \; \la A^{-1} w, z \ra \; \leq \; c_u - \la w, r \ra \\
\\
& z \; \geq \; 0 .
\ea
$$ 
Therefore, for notation simplicity, we only need to consider problem~\eqref{eq:GLS} in the following form: 
\begin{equation} \label{eq:SLS}
\ba{rcl}
\min\limits_{z \in \mathbb{R}^n}& \left\| z - b \right\|_2^2 \\
\\
\ST & c_l \; \leq \; \la w, z \ra \; \leq \; c_u \\
\\
& z \; \geq \; 0 .
\ea
\end{equation}

If $w_i = 0$ for some $i \in \{1,\dots, n \}$ in problem \eqref{eq:SLS}; then the corresponding solution of $z_i$ must be $[b_i]_+$.  After determining the solutions for these elements, we thereafter assume $w_i \neq 0$ for $i=1, \dots, n$. 
\subsubsection{Analytical Solution}
In this part, we deduce the analytical solution for our least squares problem.
\begin{theorem} \label{thm:lsqg}
The solution to \eqref{eq:SLS} is
$$
z^* \; = \; 
\left[ b 
 - \frac{ \left[ \la \tilde{w}, \tilde{b} \ra - c_u \right]_+ }{\|\tilde{w}\|^2_2} w  
+\frac{\left[ c_l - \la \tilde{w}, \tilde{b} \ra \right]_+}{\|\tilde{w}\|_2^2 }w\right]_+  ,
$$
where $\tilde{w}$ and $\tilde{b}$ denote the subvectors of $w$ and $b$ with indices in the set
$$
S \; = \; \left\{ 1 \leq i \leq n \colon b_i >
 \frac{ \left[ \la \tilde{w}, \tilde{b} \ra - c_u \right]_+ }{\|\tilde{w}\|^2_2} w_i 
-\frac{\left[ c_l - \la \tilde{w}, \tilde{b} \ra \right]_+}{\|\tilde{w}\|_2^2 }w_i 
\right\} .
$$
\end{theorem}

\begin{proof}
Because the constraints of problem~{eq:SLS} are linear, Lagrange multipliers exist.  Let's write down the  Lagrangian function:
$$
L(z, \lambda) \; = \; \left\|  z - b \right\|_2^2 + \lambda_l \left(c_l - \la w, z \ra \right)  + \lambda_u \left( \la w, z \ra - c_u \right) , \quad ( z \geq 0, \lambda_l \geq 0, \lambda_u \geq 0).
$$
The solutions to problem \eqref{eq:SLS} can be obtained by solving the following problem:
$$
\max\limits_{\lambda_l \geq 0, \lambda_u \geq 0} 
\min\limits_{z \geq 0} 
L(z, \lambda) .
$$
Note that
$$
L(z, \lambda) = \left\| z - b + \frac{\lambda_u - \lambda_l}{2} w \right\|_2^2 - \left(\frac{\lambda_u - \lambda_l}{2} \right)^2 \left\|w \right\|_2^2 + \left(\lambda_u - \lambda_l \right) \la w, b \ra + \lambda_l c_l - \lambda_u c_u, 
$$
from which we conclude that the solution to the Lagrangian dual
$ \min\limits_{z \geq 0} L(z, \lambda) $
is
$$
z^* \; = \; \left[ b - \frac{\lambda_u - \lambda_l}{2}w \right]_+.
$$
We next determine the optimal values for $\lambda_u$ and $\lambda_l$.

We first consider $\lambda_l$.

Let $S$ denote the index set
$$
S \; \defeq \; \left\{ 1 \leq i \leq n \colon b_i > \frac{\lambda_u - \lambda_l}{2} w_i \right\} .
$$

Let $\tilde{w}$ and $\tilde{b}$ denote the subvectors of $w$ and $b$ with indices in $S$.
Let $\bar{b}$ denote the subvector of $b$ with indices not in $S$.  We then have
$$
\ba{rcl}
L(z^*, \lambda) & = & - \left( \frac{\lambda_u - \lambda_l}{2} \right)^2 \left\| \tilde{w} \right\|_2^2 + \left(\lambda_u - \lambda_l \right) \la \tilde{w}, \tilde{b} \ra + \lambda_l c_l - \lambda_u c_u + \left\| \bar{b} \right\|_2^2\\
\\
& = & - \left( \frac{\|\tilde{w}\|_2}{2}\lambda_l - \frac{ \frac{\|\tilde{w}\|_2^2}{2}\lambda_u - \la \tilde{w}, \tilde{b} \ra + c_l}{\|\tilde{w}\|_2} \right)^2
+ \left( \frac{ \frac{\|\tilde{w}\|_2^2}{2}\lambda_u - \la \tilde{w}, \tilde{b} \ra + c_l}{\|\tilde{w}\|_2} \right)^2\\
\\
& & + \left\|\bar{b}\right\|_2^2 - \frac{\|\tilde{w}\|_2^2}{4}\lambda_u^2 + \lambda_u \la \tilde{w}, \tilde{b} \ra - \lambda_u c_u .
\ea
$$
Hence a solution of $\lambda_l$ for $\max_{\lambda \geq 0} L(z^*, \lambda)$ must be in the form
$$
\lambda_l^* \ = \; \left[ \lambda_u - \frac{2}{\|\tilde{w}\|_2^2}\left( \la \tilde{w}, \tilde{b} \ra - c_l \right) \right]_+ .
$$
To determine the solution of $\lambda_u$, we consider different cases.

\paragraph{Case 1.} For $\lambda_u < \frac{2}{\|\tilde{w}\|_2^2} \left( \la \tilde{w}, \tilde{b} \ra - c_l \right)$, the representation of $\lambda_l^*$ is reduced to
$$
\lambda_l^* \; = \; 0 .
$$
Since $\lambda_u \geq 0$, we have 
$$
\la \tilde{w}, \tilde{b} \ra \; > \; c_l .
$$
And
$$
\ba{rcl}
L(z^*, \lambda_l^*, \lambda_u^*) & = & -\frac{\|\tilde{w}\|_2^2}{4}\lambda_u^2 + \la \tilde{w}, \tilde{b} \ra \lambda_u - c_u \lambda_u + \|\bar{b}\|_2^2 \\
\\
& = & - \left(\frac{\|\tilde{w}\|_2}{2}\lambda_u - \frac{\la \tilde{w}, \tilde{b}\ra - c_u}{\|\tilde{w}\|_2} \right)^2 + \left\|\bar{b}\right\|_2^2 + 
\left( \frac{\la \tilde{w}, \tilde{b}\ra - c_u}{\|\tilde{w}\|_2} \right)^2 .
\ea
$$
Therefore, for this case, the solution to $\max_{\lambda \geq 0} L(z^*, \lambda)$ is
$$
\lambda_u^* \; = \; \frac{2}{\|\tilde{w}\|_2^2} \left[ \la \tilde{w}, \tilde{b}\ra - c_u \right]_+ .
$$

\paragraph{Case 1.a.} When $c_l < \la \tilde{w}, \tilde{b} \ra < c_u$, we have
$$
\lambda_u^* \; = \;0, \quad z^* \; =\; [b]_+ .
$$
 
\paragraph{Case 1.b.} When $ \la \tilde{w}, \tilde{b} \ra \geq c_u$, we have
$$
\ba{rcl}
\lambda_u^* & = & \frac{2}{\|\tilde{w}\|_2^2} \left( \la \tilde{w}, \tilde{b} \ra - c_u \right) \\
\\
z^* & = & \left[ b - \frac{\la \tilde{w}, \tilde{b} \ra - c_u }{\|\tilde{w}\|^2_2} w \right]_+ \\
\\
\la w, z^* \ra & = & c_u .
\ea
$$

\paragraph{Case 2.} For $\lambda_u \geq \frac{2}{\|\tilde{w}\|_2^2} \left( \la \tilde{w}, \tilde{b} \ra - c_l \right)$, we have
$$
\lambda_l^* \; = \; \lambda_u -  \frac{2}{\|\tilde{w}\|_2^2} \left( \la \tilde{w}, \tilde{b} \ra - c_l \right) .
$$
And
$$
L(z^*, \lambda_l^*, \lambda_u) \; = \; \left( c_l - c_u \right) \lambda_u + \left( \frac{\la \tilde{w}, \tilde{b} \ra - c_l}{\|\tilde{w}\|_2} \right)^2 
+ \left\|\bar{b} \right\|_2^2 .
$$
Therefore, in this case, the solution to $\max_{\lambda \geq 0} L(z^*, \lambda)$ is
$$
\ba{rcl}
\lambda_u^* & = &  0 \\
\\
\lambda_l^* & = & 
2 \frac{c_l - \la \tilde{w}, \tilde{b} \ra }{\|\tilde{w}\|_2^2}\\
\\
z^* & = & \left[ b + 
\frac{c_l - \la \tilde{w}, \tilde{b} \ra }{\|\tilde{w}\|_2^2 }w\right]_+ \\
\\
\la w, z^* \ra & = & c_l .
\ea
$$
Because $\lambda_l^* \geq 0$, this case implies 
$$
\la \tilde{w}, \tilde{b} \ra \; \leq \;  c_l .
$$
Combining \textbf{Case 1} and \textbf{Case 2}, we obtain
$$
z^* \; = \; 
\left[ b 
 - \frac{ \left[ \la \tilde{w}, \tilde{b} \ra - c_u \right]_+ }{\|\tilde{w}\|^2_2} w  
+\frac{\left[ c_l - \la \tilde{w}, \tilde{b} \ra \right]_+}{\|\tilde{w}\|_2^2 }w\right]_+  ,
$$
where $\tilde{w}$ denote the subvector of $w$ with indices in the set
$$
S \; = \; \left\{ 1 \leq i \leq n \colon b_i >
 \frac{ \left[ \la \tilde{w}, \tilde{b} \ra - c_u \right]_+ }{\|\tilde{w}\|^2_2} w_i 
-\frac{\left[ c_l - \la \tilde{w}, \tilde{b} \ra \right]_+}{\|\tilde{w}\|_2^2 }w_i 
\right\} .
$$
\end{proof}

\begin{remark}
In our deduction, it is obvious that for $c_l = - \infty$, we have $\lambda_l^* = 0$;
for $c_u = + \infty$, we have $\lambda_u^* = 0$.
\end{remark}

\subsubsection{Algorithm} \label{SS:lsqgalgo}
From the discussion in the previous section, we know that to find the optimal solution $z^*$ of our least squares problem, we only need to determine the set $S$.
In this part, we describe how to find the set $S$ for our solution.

\paragraph{Properties of $S$ based on Lagrange multipliers}
We first give some simple observations which will be used later on.
\begin{proposition} \label{prop:sumadd}
Let $r_1 \in  \mathbb{R}$, $r_3 \in \mathbb{R}$, $r_2 > 0$, $r_4 > 0$.  Then
$$
\ba{rcl}
\frac{r_1}{r_2} \; > \; \frac{r_3}{r_4} & \Leftrightarrow 
\frac{r_1}{r_2} \; > \; \frac{r_3+r_1}{r_4+r_2} , \\
\\
\frac{r_1}{r_2} \; < \; \frac{r_3}{r_4} & \Leftrightarrow 
\frac{r_1}{r_2} \; < \; \frac{r_3+r_1}{r_4+r_2} . \\
\ea
$$
\end{proposition}

We next give some properties of the set $S$ based on Lagrange multipliers.
Observe that $\lambda_l^*$ and $\lambda_u^*$ cannot be both positive at the same time.  We organize our analysis based on scenarios depending on the signs of the Lagrange multipliers.

\paragraph{Case 1.} $\lambda_u^* > 0$.

By the deduction above and Lagrange multiplier properties, we have the corresponding relations:
$$
\ba{rcl}
\la w, z \ra & = & c_u \\
\\
\la \tilde{w}, \tilde{b} \ra & > & c_u \\
\\
\lambda_u^* & = & 2 \frac{ \la \tilde{w}, \tilde{b} \ra - c_u}{\|\tilde{w}\|_2^2} .
\ea
$$
We next consider which indices are in the set $S$.
\begin{enumerate}
\item
$ S_1 \defeq \{i \colon w_i > 0, b_i \geq 0\}$: 

\begin{lemma}
Suppose $\frac{b_j}{w_j} \geq \frac{b_i}{w_i}$.  If $i \in S$; then $j \in S$ as well.  
\end{lemma}

\begin{proof}
Assume $j \notin S$.
Since $i \in S$, we have
$$
\frac{b_j w_j}{w_j^2} \; \geq  \; 
\frac{b_i w_i}{w_i^2} \; >  \; 
\frac{\la \tilde{w}, \tilde{b} \ra - c_u}{\|\tilde{w}\|_2^2} .
$$
By Proposition~\ref{prop:sumadd}, we have
$$
\frac{b_j w_j}{w_j^2} \; >  \; 
\frac{\la \tilde{w}, \tilde{b} \ra - c_u + b_j w_j}{\|\tilde{w}\|_2^2 + w_j^2} .
$$
Therefore, $j \in S$.
\end{proof}

\item
$ S_2 \defeq \{i \colon w_i > 0, b_i < 0\}$: 

By the definition of $S$, we have $S_2 \nsubseteq S$.

\item
$ S_3 \defeq \{i \colon w_i < 0, b_i \geq 0\}$: 

It is obvious $S_3 \subseteq S$.

\item
$ S_4 \defeq \{i \colon w_i < 0, b_i < 0\}$: 

\begin{lemma}
Suppose $\frac{b_j}{w_j} \leq \frac{b_i}{w_i}$.  If $i \in S$; then $j \in S$ as well.  
\end{lemma}

\begin{proof}
Assume $j \notin S$.
Since $i \in S$, we have
$$
\frac{b_j w_j}{w_j^2} \; \leq  \; 
\frac{b_i w_i}{w_i^2} \; <  \; 
\frac{\la \tilde{w}, \tilde{b} \ra - c_u}{\|\tilde{w}\|_2^2} .
$$
By Proposition~\ref{prop:sumadd}, we have
$$
\frac{b_j w_j}{w_j^2} \; <  \; 
\frac{\la \tilde{w}, \tilde{b} \ra - c_u + b_j w_j}{\|\tilde{w}\|_2^2 + w_j^2} .
$$
Therefore, $j \in S$.
\end{proof}

\end{enumerate}

\paragraph{Case 2.} $\lambda_l^* > 0$.

For this case, we have
$$
\ba{rcl}
\la w, z^* \ra & = & c_l \\
\\
\la \tilde{w}, \tilde{b} \ra & < & c_l  \\
\\
\lambda_l^* & = & 
2 \frac{c_l - \la \tilde{w}, \tilde{b} \ra }{\|\tilde{w}\|_2^2} .
\ea
$$
We now determine which indices are in the set $S$.
\begin{enumerate}
\item
$ S_1 \defeq \{i \colon w_i > 0, b_i \geq 0\}$: 

By the definition of $S$, we have $S_1 \subseteq S$.

\item
$ S_2 \defeq \{i \colon w_i > 0, b_i < 0\}$: 

Similar to the case for $\lambda_u^* > 0$, we have:

Suppose $\frac{b_j}{w_j} \geq \frac{b_i}{w_i}$.  If $i \in S$; then $j \in S$ as well.  

\item
$ S_3 \defeq \{i \colon w_i < 0, b_i \geq 0\}$: 

Similar to the case for $\lambda_u^* > 0$, we have:

Suppose $\frac{b_j}{w_j} \leq \frac{b_i}{w_i}$.  If $i \in S$; then $j \in S$ as well.  

\item
$ S_4 \defeq \{i \colon w_i < 0, b_i < 0\}$: 

By the definition of $S$, we have $S_4 \nsubseteq S$.

\end{enumerate}

\paragraph{Case 3.} $\lambda_l^* = \lambda_u^* = 0$.

For this case, we have 
$$
Z^* \; = \; [b]_+ .
$$

\paragraph{Determine the signs of Lagrange multipliers}

We next show that whether the Lagrange multiplier is positive or not can be determined by $\la w, [b]_+ \ra$.

\begin{lemma}
The Lagrange multiplier $\lambda_l^*$ satisfies the following condition:
$$
\lambda_l^* \;  \; \begin{cases} \; = \; 0  & \la w, [b]_+ \ra \geq c_l, \\
\; > \; 0 & \la w, [b]_+ \ra < c_l.
\end{cases}
$$
\end{lemma}
\begin{proof}
We first use contradiction to prove the result for the case $\la w, [b]_+ \ra \geq c_l$.
Assume $\lambda_l^* > 0$.
By the properties for $\lambda_l^* > 0$,  we have $c_l > \la \tilde{w}, \tilde{b} \ra$ and $S_1 \subseteq S$.
Since $\la w, [b]_+ \ra \geq c_l$, we must have $S_2 \cap S \neq \emptyset$.

 Let $l \in S_2 \cap S$ such that $\frac{b_l}{w_l} \cdot \frac{w_j}{b_j} \geq 1$ ($\forall \, j \in S_2 \cap S$). We would have
$$
\ba{rcl}
\sum\limits_{j \in S_2 \cap S}  \frac{b_l}{w_l} w_j^2 \; = \;
\sum\limits_{j \in S_2 \cap S} \left( \frac{b_l}{w_l} \cdot \frac{w_j}{b_j} \right) w_j b_j & \leq &
\sum\limits_{j \in S_2 \cap S} w_j b_j ,\\
\\
0 & \leq &  \sum\limits_{j \in S \setminus S_2} w_j b_j - c_l \; = \; \la w, [b]_+ \ra - c_l  . \\
\ea
$$
Adding the above two inequalities together, we would have
$$
\ba{rcl}
-\frac{b_l}{w_l} & \geq & \frac{ c_l - \sum_{j \in S} w_j b_j}{\sum_{j \in S_2 \cap S} w_j^2} \; \geq \;  \frac{1}{2}\lambda_l^* ,
\ea
$$
contradicting to $ l \in S$.

We next consider the case $\la w, [b]_+ \ra < c_l$.

By the assumption, we have
$$
\ba{rcl}
c_l - \la w, [b]_+ \ra \; = \; 
c_l - \sum\limits_{i \in S \cap (S_1 \cup S_3)} w_j b_j & > & 0 \\
\\
-\sum\limits_{i \in S \cap S_2 } w_j b_j & \geq & 0 .\\
\ea
$$
Adding the above two inequalities together, we have
$$
\lambda_l^* \; > \; 0 .
$$
\end{proof}

Similarly, we have the results for $\lambda_u^*$.
\begin{lemma}
The Lagrange multiplier $\lambda_u^*$ satisfies the following condition:
$$
\lambda_u^* \;  \; \begin{cases} \; = \; 0  & \la w, [b]_+ \ra \leq c_u, \\
\; > \; 0 & \la w, [b]_+ \ra > c_u.
\end{cases}
$$
\end{lemma}

For the case $\lambda_u^* > 0$, deleting any index from the set $S_1 \cap S$ decreases the value $\frac{\la \tilde{w}, \tilde{b} \ra - c_u}{\|\tilde{w}\|_2^2}$, and deleting any index from the set $S_4 \cap S$ increases that value.
Similarly, for the case $\lambda_l^* > 0$, deleting any index from the set $S_2 \cap S$ decreases the value $\frac{\la \tilde{w}, \tilde{b} \ra - c_l}{\|\tilde{w}\|_2^2}$, and deleting any index from the set $S_3 \cap S$ increases that value.

The discussion above proves that our algorithm below finds an optimal solution of the problem \eqref{eq:GLS}. 
\paragraph{Algorithm} 
Reduce problem~\eqref{eq:GLS} to the form~\eqref{eq:SLS} and solve problem~\eqref{eq:SLS}:

Let $n_i$ be the cardinality of the index set $S_i$, ($i=1, \dots, 4$). 
We first compute $\la w, [b]_+ \ra$.

\begin{itemize}
\item
If $\la w, [b]_+ \ra \in [c_l, c_u]$, we let
$$
z^* \; = \; [b]_+ .
$$

\item
If $\la w, [b]_+ \ra > c_u$, we do the following.

\begin{enumerate}
\item
Re-order the elements in $S_1$ so that
$$
b_{\sigma(1)}/w_{\sigma(1)} \geq b_{\sigma(2)}/w_{\sigma(2)}
\geq \cdots b_{\sigma(n_1)}/w_{\sigma(n_1)} .
$$
Re-order the elements in $S_4$ so that
$$
b_{\tau(1)}/w_{\tau(1)} \leq b_{\tau(2)}/w_{\tau(2)}
\leq \cdots b_{\tau(n_1)}/w_{\tau(n_4)} .
$$

\item
Let 
$$
S = S_3, \quad  T = \sum_{i \in S_3} w_i b_i  - c_u,  \quad v = \sum_{i \in S_3} w_i^2, \quad j = 1, \quad l = 1.
$$ 

\item
Repeat the following two \textit{while} loops till stable.

While $v \frac{b_{\sigma(j)}}{w_{\sigma(j)}} > T$ and $j \leq n_1$ , do 
$$
S \cup \{\sigma(j)\} \rightarrow S, \quad T + w_{{\sigma(j)}}
b_{\sigma(j)}  \rightarrow T, \quad v + w_{\sigma(j)}^2 \rightarrow v, \quad
j+1 \rightarrow j . \quad
$$

While $v \frac{b_{\tau(l)}}{w_{\tau(l)}} < T$ and $l \leq n_4$, do
$$
S \cup \{\tau(l)\} \rightarrow S, \quad T + w_{{\tau(l)}}
b_{\tau(l)}  \rightarrow T, \quad v + w_{\tau(l)}^2 \rightarrow v, \quad
l+1 \rightarrow l . \quad
$$

\item
Let
$$
\ba{rcl}
z_i^* & = & \begin{cases}
0 & i \in \bar{S} \\
b_i -  \frac{T}{v} w_i & i \in S .
\end{cases}
\ea
$$
\end{enumerate}

\item
If $\la w, [b]_+ \ra < c_l$, we do the following.

\begin{enumerate}
\item
Re-order the elements in $S_2$ so that
$$
b_{\sigma(1)}/w_{\sigma(1)} \geq b_{\sigma(2)}/w_{\sigma(2)}
\geq \cdots b_{\sigma(n_2)}/w_{\sigma(n_2)} .
$$
Re-order the elements in $S_3$ so that
$$
b_{\tau(1)}/w_{\tau(1)} \leq b_{\tau(2)}/w_{\tau(2)}
\leq \cdots b_{\tau(n_3)}/w_{\tau(n_3)} .
$$

\item
Let 
$$
S = S_1, \quad T = \sum_{i \in S_1} w_i b_i  - c_l, \quad v = \sum_{i \in S_1} w_i^2, \quad j = 1, \quad l = 1.
$$ 

\item
Repeat the following two while loops till stable. 
\begin{enumerate}
\item
While $v \frac{b_{\sigma(j)}}{w_{\sigma(j)}} > T$ and $j \leq n_2$, let
$$
S \cup \{\sigma(j)\} \rightarrow S, \quad T + w_{{\sigma(j)}}
b_{\sigma(j)}  \rightarrow T, \quad v + w_{\sigma(j)}^2 \rightarrow v, \quad
j+1 \rightarrow j . \quad
$$
\item
While $v \frac{b_{\tau(l)}}{w_{\tau(l)}} < T$ and $l \leq n_3$, let
$$
S \cup \{\tau(l)\} \rightarrow S, \quad T + w_{{\tau(l)}}
b_{\tau(l)}  \rightarrow T, \quad v + w_{\tau(l)}^2 \rightarrow v, \quad
l+1 \rightarrow l . \quad
$$
\end{enumerate}
\item
Let
$$
\ba{rcl}
z_i^* & = & \begin{cases}
0 & i \in \bar{S} \\
b_i -  \frac{T}{v} w_i & i \in S .
\end{cases}
\ea
$$

\end{enumerate}

\end{itemize}

\begin{lemma}
After reducing problem \eqref{eq:GLS} to problem \eqref{eq:SLS}, the algorithm above stops at an optimal solution to \eqref{eq:SLS} with at most $n^2+14n+1$ arithmetic operations and $2n+3$ auxiliary storage space units.  If all $w_i=1$, the above algorithm needs at most $n^2+7n+1$ arithmetic operations and $n+2$ auxiliary storage space units.
\end{lemma}
\begin{proof}
Determining the signs of $b_i$ and computing $\la w, [b]_+\ra$ takes $3n-1$ flops.
Further dividing the index set into $S_1, \dots, S_4$ takes another $n$ flops. 
Comparing $\la w, [b]_+\ra$ with $c_l$ and $c_u$ takes $2$ operations.
Computing $b_i/w_i$ ($i=1,\dots, n$) takes $n$ flops. Bubble sorting the elements in the sets $S_1, \dots, S_4$ takes at most $n(n-1)$ operations.
Two auxiliary vectors of size $n$ are required to store $b_j/w_j$ for ($j=1, \dots, n$) and the sorted index set.
The number of flops needed for \textit{Step 2} and \textit{Step 3} is at most $7n$.  We also need three auxiliary space units to store $j$, $v$ and $T$.
\textit{Step 4} takes at most $3n$ flops. Since we overwrite $b$ by $z$, we don't need an additional vector for $z$.
Therefore, at most a total of $n^2+14n+1$ operations and $2n+3$ auxiliary storage space units are required for our algorithm.
If all $w_i=1$, we don't need to divide and multiply the intermediate results by $w_j$.  The index sets $S_3$ and $S_4$ are not needed.  And $b_j/w_j$ doesn't need to be stored.
 As well, we don't need to keep and compute $v$, since its value equals to $j$.
Therefore, the total number of operations is reduced to at most $n^2 + 7n+1$.

\end{proof}

\subsection{Symmetric Matrix Projection with Lower Bounds and a Two-Sided Linear Constraint} \label{SS:MProj}
\begin{theorem} \label{thrm:projH}
For given $U \in \mathbb{C}^n$, and $c_l, c_u, r \in
\mathbb{R}$ with $c_u \geq \max \{n r, c_l\}$, the solution
$\hat{Z}$ to the projection problem \eqref{eq:projH}
is the following.

Let $Q\Lambda Q^*$ be the eigenvalue decomposition of
$\frac{U+U^*}{2}$.
Let $\lambda$ denote the diagonal entries of $\Lambda$.

Denote
$$
S_0 \; \defeq \; \left\{ 1 \leq j \leq n \colon \lambda_j \leq r \right\} , \quad
\bar{S}_0 \; \defeq \; \{1, \dots, n\} \setminus S_0 .
$$

\begin{enumerate}
\item
Assume $c_l \leq \sum_{i \in \bar{S}_0} \lambda_i + | S_0| r \leq c_u$.

Then we let
$$
\ba{rcl}
\hat{\omega}_i & = & \lambda_i \quad i \in \bar{S}_0 , \qquad \hat{\omega}_i \; = \; r \quad i \in S_0 .
\ea
$$
\item
Assume $\sum_{i \in \bar{S}_0} \lambda_i + |S_0| r > c_u$.

Then there is a partition of $\bar{S}_0$ as $\bar{S}_0 = S \cup \bar{S}$:
$$
\ba{rcl}
S & \defeq & \left\{  i \in \bar{S}_0 \colon \lambda_i > \frac{\sum_{j \in S} \lambda_j +  nr - c_u}{|S|} \right\}, \\
\\
\bar{S} & \defeq & \left\{ i \in \bar{S}_0 \colon \lambda_i \leq \frac{\sum_{j \in S} \lambda_j + nr - c_u}{|S|} \right\}.
\ea
$$

And we let
$$
\ba{rcl}
\hat{\omega}_i & = & \begin{cases}
r & i \in \bar{S} \cup S_0 \\
\lambda_i -  \frac{\sum_{j \in S} \lambda_j + n r - c_u}{|S|} + r & i \in S .
\end{cases}
\ea
$$

\item
Assume $\sum_{i \in \bar{S}_0} \lambda_i + |S_0| r < c_l$.

Then there is a partition of $S_0$ as $S_0 = S_l \cup \bar{S}_l$ where
$$
\ba{rcl} S_l & = & \left\{ i \in S_0 \colon \frac{c_l - \sum_{j \in
S_l \cup  \bar{S}_0} \lambda_j - nr}{|\bar{S}_0| + |S_l|} > -
\lambda_i \right\} . \ea
$$

We let
$$
\ba{rcl}
\hat{\omega}_i & = & \begin{cases}
\lambda_i + \frac{c_l - \sum_{j \in \bar{S}_0 \cup S_l} \lambda_j - |\bar{S}_l| r}{|\bar{S}_0|+|S_l|}  & i \in \bar{S}_0 \cup S_l \\
 r & i \in \bar{S}_l .
\end{cases}
\ea
$$

\end{enumerate}

Let $\hat{\Omega}$ be the diagonal matrix with diagonal entries $\hat{\omega}$.  Then
$\hat{Z} = Q\hat{\Omega}Q^*$ is the unique solution to \eqref{eq:projH}.

If $U \in \mathcal{S}^n$,
$\hat{Z}$ can be obtained in $(10n^3+ 3n^2+9n+5)$ flops with an auxiliary storage vector of size $(n^2+3n+4)$.
\end{theorem}
\begin{proof}
Since $Z \in \Herm^n$, we have
$$
\ba{rcl}
\|Z - U\|_F^2 & = & \frac{1}{2} \left( \|Z - U\|_F^2 + \|Z - U^*\|_F^2 \right) \\
\\
& = & \tr(Z^2) + \tr(UU^*) - \tr(ZU+ZU^*) \\
\\
& = & \tr \left( Z - \frac{U+U^*}{2} \right)^2 + \frac{1}{2}\tr(UU^*) - \frac{1}{4}\tr(U^2) - \frac{1}{4}\tr({U^*}^2) .
\ea
$$
Therefore, the solution to \eqref{eq:projH} is the same as the
solution to the following problem:
$$
\ba{rc}
\min\limits_{Z \in \Herm^n} & \|Z - \frac{U+U^*}{2}\|_F^2 \\
\ST & c_l \leq \tr(Z) \leq c_u \\
& \lambda_{\min}(Z) \geq r .
\ea
$$
Let $\hat{F}$ be the optimal value of the above problem.

By Theorem~\ref{thm:lsqg} in the Appendix, $\hat{\omega}$ in the statement of the
theorem is the solution to
$$
\begin{array}{ll}
\min\limits_{\omega \geq r} & \left\| \omega - \lambda \right\|_2 \\
\ST & c_l \leq \sum\limits_{i=1}^n \omega_i \leq c_u .
\end{array}
$$

The Hoffman-Wielandt theorem~\cite{MR0052379} states that for two
Hermitian matrices $V$ and $W$, let $\lambda_1(V), \dots,
\lambda_n(V)$ and $\lambda_1(W), \dots, \lambda_n(W)$ be the
eigenvalues of  $V$ and $W$ in non-increasing order.   Then there is
a permutation $\sigma(i)$ ($ i = 1, \dots, n $) such that
$$
\ba{rcl}
\sum_{i=1}^n \left[ \lambda_{\sigma(i)}(W) -  \lambda_{i}(V) \right]^2 & = & \|W - V \|_F^2 .
\ea
$$
And it is obvious from Theorem~\ref{thm:lsqg} in the Appendix that $\hat{\omega}$
is in the same order as $\lambda$; i.e. if $\lambda$ is arranged in
non-increasing order, $\hat{\omega}$ is also in non-increasing
order. Therefore,
$$
\ba{rcl}
\hat{F} & \geq & \| \hat{\omega} - \lambda \|_2^2 .
\ea
$$
Since $\hat{Z}$ and $\frac{U+U^*}{2}$ are unitary similar, we have
$$
\ba{rcl}
\left\| \hat{Z} - \frac{U+U^*}{2}\right\|_F^2 & = & \| \hat{\omega} - \lambda \|_F^2 .
\ea
$$
Hence $\hat{Z}$ is the solution to \eqref{eq:projH}.

Now we consider the complexity and memory requirement of getting the solution $\hat{Z}$ when $U$ is real symmetric.

The eigenvalue decomposition of $U$  by the symmetric QR algorithm takes roughly $9n^3$ flops.  Since we can overwrite $U$, $n^2$ space units are needed to store the orthogonal matrix $Q$ and about $2n+1$ auxiliary space units are needed to store intermediate results. 
The algorithm in \S\S \ref{SS:lsqgalgo} of the Appendix can be used to compute $\hat{\omega}$.
Since all the $r_i$ are identical, variable transformations from $c_l$ and $c_u$ to $\tilde{c}_l$ and $\tilde{c}_u$ takes $4$ flops, instead of $2n$ flops for $r_i$'s being heterogenous.  Therefore, calculating $\hat{\omega}$ takes at most $(n^2+9n+5)$ flops and $3n+4$ auxiliary storage space units.
Computing $Q\hat{\Omega}Q^*$ takes $(n^2(n+1) + n^2)$ flops.
Since the auxiliary vector for storing the intermediate results of the eigenvalue decomposition of $U$ can be over-written, the total length of the auxiliary vectors is $(n^2+3n+4)$. 
And the total number of flops is $(10n^3+3n^2+9n+5)$ for $U \in \mathcal{S}^n$.

\end{proof}

If $n \leq 3$, the characteristic polynomial of $\frac{U+U^*}{2}$ is of order no more than $3$; therefore, its eigenvalues can be obtained analytically. Its eigenvectors can then be obtained by solutions to its eigen-systems. 

\section*{Appendix: Updating the Parameters}
As is stated earlier, by \cite[Theorem 1]{Nes07}, the duality gap of the $t$th iteration generated by the primal-dual algorithm is bounded by
\begin{equation} \label{eq:deltabd}
\frac{1}{\sum_{l=0}^t \alpha_t} \delta_{t}, \quad
\text{ with } \; \delta_{t} \; \leq \; \beta_{t+1} D + \frac{1}{2} \sum_{l=0}^t \frac{\alpha_l^2}{\beta_l} \|g_l\|_{*}^2 .
\end{equation}
In our algorithm, $\|g_l\|_* = \|[(g_E)_l, (g_x)_l]\|_*$, $D= \tau D_E + (1-\tau) D_x$.

For $t =1, \dots$, let :
\begin{equation} \label{eq:betahat}
\hat{\beta}_0 \; = \; \hat{\beta}_1 \; = \; 1, \quad  \hat{\beta}_{t+1} \; = \; \hat{\beta}_t + \frac{1}{\hat{\beta}_t}.
\end{equation}
And
$$
\beta_t \; = \;\sigma \hat\beta_t .
$$
\paragraph{Simple Dual Averages}
$$\alpha_t \; = \; 1. $$
Assume $\|g_t\|_{*} \leq L$ for $t=1, \dots $; then by \cite[Theorem 2]{Nes07}, we have
$$
\delta_{t} \; \leq \; \hat\beta_{t+1} \left( D\sigma + \frac{1}{2\sigma}L^2 \right), \quad \sum_{l=0}^t \alpha_l = t+1.
$$
\paragraph{Weighted Dual Averages}
$$\alpha_t \; = \; \frac{1}{\|g_t\|_*} . $$
Assume $\|g_t\|_{*} \leq L$ for $t=1, \dots $; then by \cite[Theorem 3]{Nes07}, we have
$$
\delta_{t} \; \leq \; \hat\beta_{t+1} \left( D\sigma + \frac{1}{2\sigma} \right), \quad \sum_{l=0}^t \alpha_l \; \geq \; \frac{t+1}{L}.
$$

The above results show that the convergence rate of the algorithm depends on the choice of $\sigma$. 
It is not possible to determine the optimal $\sigma$ without the knowledge of $D$ or $L$.
In this part, we show how to dynamically update the parameter $\sigma_t$ in the algorithm to obtain the best convergence rate. 

\paragraph{Choosing $\beta_t$:} Let $\sigma_0 > 0$ be the smallest possible value for $\sigma$.  Let $w > 0$ be the number of steps for each test in updating $\sigma$.
\begin{center}
\fbox{
\begin{minipage}{0.7\textwidth}
\begin{enumerate}
\item
Choose $w > 0, \sigma_0 > 0 $.
\item
Let 
$$
v \; = \; 0 , \qquad
\sigma \; = \; \sigma_0.
$$
For $t = 0, \dots w$, let
$$ \beta_t  =  \sigma_0 \hat{\beta}_t . $$
\item
Repeat the following until convergence rate starts to decrease.

\begin{itemize}
\item
Let 
$$
 v \; = \; v+1, \qquad \sigma \; = \; 2*\sigma .
$$
\item
For $t= vw+1 \dots (v+1)w$, let 
$$ \beta_t \; = \; \sigma \hat{\beta}_t .$$
\end{itemize}

\item
Let 
$$ v \; = \; v-1, \qquad \sigma \; = \; \sigma/2.$$
For $t=(v+2)w+1, \dots$, let
$$ \beta_t \; = \; \sigma \hat{\beta}_t .$$

\end{enumerate}
\end{minipage}
}
\end{center}

\begin{theorem}
The total number of test steps for the above procedure of determining $\sigma$ is finite.
And the total number of iterations of the algorithm including the  above procedure is at most $5/3$ of the algorithm without the procedure but using optimal parameters plus a term in the order of $\mathcal{O}\big(\frac{1}{\epsilon}\big)$.
\end{theorem}
\begin{proof}
Assume that at iteration $t$ we have obtained the $\sigma$ from the above procedure.  Denote $v = v_t$.  Suppose $\|g_l\|_{*} \leq L$ ($l = 0, \dots, t$).   Since there is one backtrack period with $w$ steps before landing at the current $\sigma$,
from the above procedure, we have $\sigma = 2^{v_t} \cdot \sigma_0$ and $\beta_t = \sigma \hat{\beta}_t$.

To prove the theorem, we need to bound $\delta_t$.

We first consider the \textit{method of simple dual averages}.
By \eqref{eq:deltabd}, 
$$ 
\ba{rcl}
\delta_{t(s)} & \leq & \
\sigma_0 2^{v_t}\hat\beta_{t+1} D
+ \sum\limits_{v=0}^{v_t+1} \frac{L^2}{\sigma_02^{v+1}} \sum\limits_{l=v \cdot w +1}^{(v+1)w} \frac{1}{\hat{\beta}_l}
+ \sum\limits_{l=(v_t+2)w + 1}^t \frac{L^2}{2^{v_t+1} \sigma_0 \hat{\beta}_l} \\
\\
& = & 
\sigma_0 2^{v_t}\hat\beta_{t+1} D
+ \frac{L^2}{\sigma_0 2^{v_t+1}} \sum\limits_{l=0}^t \frac{1}{\hat\beta_l}
+ \frac{L^2}{\sigma_02^{v_t+1}} \sum\limits_{v=1}^{v_t-1} \sum\limits_{l=0}^{v \cdot w} \frac{1}{\hat\beta_l}  - \frac{L^2}{\sigma_0 2^{v_t+2}} \sum\limits_{l=(v_t+1)w+1}^{(v_t + 2)w} \frac{1}{\hat\beta_l}\\
\\
& \stackrel{\rm\eqref{eq:betahat}}{=} & 
\hat\beta_{t+1} \left( \sigma_0 2^{v_t} D + \frac{1}{\sigma_0 2^{v_t+1}} L^2 \right)
+ \frac{L^2}{\sigma_0 2^{v_t+1}} \sum\limits_{v=1}^{v_t-1} \hat\beta_{v \cdot w + 1} \\
\\
& & + \frac{L^2}{\sigma_0 2^{v_t+2}} \left[\hat\beta_{ (v_t+1)w + 1} - \hat\beta_{ (v_t+2)w + 1}\right] .

\ea
$$

To further estimate the bound, we use \cite[Lemma 3]{Nes07}:
$$
\hat\beta_t \; \leq \; \frac{1}{1+\sqrt{3}} + \sqrt{2t-1}, \quad t \geq 1 .
$$
From the above result, we have
$$
\ba{rcl}
\sum\limits_{v=1}^{v_t-1} \hat\beta_{v \cdot w + 1} & \leq & 
\frac{v_t-1}{1+\sqrt{3}} +  \sum\limits_{v=1}^{v_t-1}\sqrt{2vw+1} \\
\\
& \leq & 
\frac{v_t-1}{1+\sqrt{3}} +  \sqrt{\frac{1}{v_t-1} \sum\limits_{v=1}^{v_t-1} (2vw+1)} \\
\\
& = & \frac{v_t-1}{1+\sqrt{3}} + \sqrt{v_tw+1} \\
\\
& \leq & 2^{v_t} \sqrt{w/2}.
\ea
$$

The optimal value of $\sigma$ is $\sigma^* = \frac{L}{\sqrt{2D}}$.  The total number of iterations decreases with $\sigma$ for $\sigma < \sigma^*$ and increases with $\sigma$ for $\sigma > \sigma^*$.
Therefore, we have
$$
\ba{rcl}
v_t & \leq & \frac{1}{2} + \log_2 \frac{L}{\sigma_0 \sqrt{D}} \\
\\
\frac{\sigma^*}{2} \; \leq \; \sigma & \leq & 2 \sigma^* .
\ea
$$
From the above inequalities, we obtain that the total number of test steps for the \textit{method of simple dual averages} to obtain an optimal $\sigma$ is no more than
$ \lceil \frac{5}{2} + \log_2 \frac{L}{\sigma_0 \sqrt{D}} \rceil w$.  And 
$( D\sigma + \frac{1}{2\sigma}L^2 )
/( D\sigma^* + \frac{1}{2\sigma^*} L^2 ) \leq 5/3$.
Therefore, the total number of iterations of our procedure for the \textit{method of simple dual averages} is at most $5/3$ of that with optimal parameter plus $\mathcal{O}\big(\frac{\sqrt{w}L^2}{2\sqrt{2}\sigma_0\epsilon}\big)$.

Similarly, for the \textit{method of weighted dual averages}, we have
$$
\ba{rcl}
\delta_{t(d)} & \leq & 
\hat\beta_{t+1} \left( \sigma_0 2^{v_t} D + \frac{1}{\sigma_0 2^{v_t+1}} \right)
+ \frac{\sqrt{w}}{\sigma_0 2\sqrt{2}}.
\ea
$$
The optimal value of $\sigma$ is $\sigma^* = \frac{1}{\sqrt{2D}}$.
Therefore, we obtain
$$
\ba{rcl}
v_t & \leq & \frac{1}{2} - \log_2 \sigma_0 \sqrt{D} \\
\\
\frac{\sigma^*}{2} \; \leq \; \sigma & \leq & 2 \sigma^* .
\ea
$$
Since
$$
\sum_{l=0}^t \; \geq \; \frac{t+1}{L},
$$
we conclude that the total number of test steps for the \textit{method of weighted dual averages} to obtain an optimal $\sigma$ is no more than
$ \lceil \frac{5}{2} - \log_2 \sigma_0 \sqrt{D} \rceil w$.  And 
$( D\sigma + \frac{1}{2\sigma} )
/( D\sigma^* + \frac{1}{2\sigma^*} ) \leq 5/3$.
Therefore, the total number of iterations of our procedure for the \textit{method of weighted dual averages} is at most $5/3$ of that by the original algorithm with optimal parameter plus $\mathcal{O}\big(\frac{\sqrt{w}L}{2\sqrt{2}\sigma_0\epsilon}\big)$.

\end{proof}
The worst case complexity bound of the original algorithm is $\mathcal{O}(\frac{1}{\epsilon^2})$ \cite{Nes07}.  Since our procedure adds a term of $\mathcal{O}(\frac{1}{\epsilon})$, the complexity remains at $\mathcal{O}(\frac{1}{\epsilon^2})$.

{
\bibliographystyle{plain}
\bibliography{grad}

\begin{thebibliography}{10}

\bibitem{MR0108399}
Kenneth~J. Arrow, Leonid Hurwicz, and Hirofumi Uzawa.
\newblock {\em Studies in linear and non-linear programming}.
\newblock With contributions by H. B. Chenery, S. M. Johnson, S. Karlin, T.
  Marschak, R. M. Solow. Stanford Mathematical Studies in the Social Sciences,
  vol. II. Stanford University Press, Stanford, Calif., 1958.

\bibitem{MR1724765}
A.~Ben-Tal, M.~Ko{\v{c}}vara, A.~Nemirovski, and J.~Zowe.
\newblock Free material design via semidefinite programming: the multiload case
  with contact conditions.
\newblock {\em SIAM J. Optim.}, 9(4):813--832 (electronic), 1999.
\newblock Dedicated to John E. Dennis, Jr., on his 60th birthday.

\bibitem{MR1327483}
M.~P. Bends{\o}e, J.~M. Guedes, R.~B. Haber, P.~Pedersen, and J.~E. Taylor.
\newblock An analytical model to predict optimal material properties in the
  context of optimal structural design.
\newblock {\em Trans. ASME J. Appl. Mech.}, 61(4):930--937, 1994.

\bibitem{MR2782122}
Antonin Chambolle and Thomas Pock.
\newblock A first-order primal-dual algorithm for convex problems with
  applications to imaging.
\newblock {\em J. Math. Imaging Vision}, 40(1):120--145, 2011.

\bibitem{CL14}
S.~Czarnecki and T.~Lewiński.
\newblock A stress-based formulation of the free material design problem with
  the trace constraint and multiple load conditions.
\newblock {\em Structural and Multidisciplinary Optimization}, 49(5):707--731,
  2014.

\bibitem{DGN14}
Olivier Devolder, François Glineur, and Yurii Nesterov.
\newblock First-order methods of smooth convex optimization with inexact
  oracle.
\newblock {\em Mathematical Programming}, 146(1-2):37--75, 2014.

\bibitem{ET99}
Ivar Ekeland and Roger T{\'e}man.
\newblock {\em Convex Analysis and Variational Problems}.
\newblock Society for Industrial and Applied Mathematics, Philadelphia, PA,
  USA, 1999.

\bibitem{HKLS10}
J.~Haslinger, M.~Kočvara, G.~Leugering, and M.~Stingl.
\newblock Multidisciplinary free material optimization.
\newblock {\em SIAM Journal on Applied Mathematics}, 70(7):2709--2728, 2010.

\bibitem{MR1261420}
Jean-Baptiste Hiriart-Urruty and Claude Lemar{\'e}chal.
\newblock {\em Convex analysis and minimization algorithms. {I}}, volume 305 of
  {\em Grundlehren der Mathematischen Wissenschaften [Fundamental Principles of
  Mathematical Sciences]}.
\newblock Springer-Verlag, Berlin, 1993.
\newblock Fundamentals.

\bibitem{MR0052379}
A.~J. Hoffman and H.~W. Wielandt.
\newblock The variation of the spectrum of a normal matrix.
\newblock {\em Duke Math. J.}, 20:37--39, 1953.

\bibitem{KSZ08}
Michal Ko{\v{c}}vara, Michael Stingl, and Jochem Zowe.
\newblock Free material optimization: recent progress.
\newblock {\em Optimization}, 57, 2008.

\bibitem{MR2112984}
Arkadi Nemirovski.
\newblock Prox-method with rate of convergence {$O(1/t)$} for variational
  inequalities with {L}ipschitz continuous monotone operators and smooth
  convex-concave saddle point problems.
\newblock {\em SIAM J. Optim.}, 15(1):229--251 (electronic), 2004.

\bibitem{MR2166537}
Yu. Nesterov.
\newblock Smooth minimization of non-smooth functions.
\newblock {\em Math. Program.}, 103(1, Ser. A):127--152, 2005.

\bibitem{Nes13}
Yu. Nesterov.
\newblock Gradient methods for minimizing composite functions.
\newblock {\em Mathematical Programming}, 140(1):125--161, 2013.

\bibitem{Nes15}
Yu~Nesterov.
\newblock Universal gradient methods for convex optimization problems.
\newblock {\em Mathematical Programming}, 152(1-2):381--404, 2015.

\bibitem{NesDualExtra}
Yurii Nesterov.
\newblock Dual extrapolation and its applications to solving variational
  inequalities and related problems.
\newblock {\em Mathematical Programming}, 109(2-3):319--344, 2007.

\bibitem{Nes07}
Yurii Nesterov.
\newblock Primal-dual subgradient methods for convex problems.
\newblock {\em Math. Program.}, 120(1, Ser. B):221--259, 2009.

\bibitem{Nes11}
Yurii Nesterov.
\newblock Barrier subgradient method.
\newblock {\em Mathematical Programming}, 127(1):31--56, 2011.

\bibitem{Ringertz}
U.T. Ringertz.
\newblock On finding the optimal distribution of material properties.
\newblock {\em Structural optimization}, 5(4):265--267, 1993.

\bibitem{SKL092}
M.~Stingl, M.~Ko\v{c}vara, and G.~Leugering.
\newblock Free material optimization with fundamental eigenfrequency
  constraints.
\newblock {\em SIAM Journal on Optimization}, 20(1):524--547, 2009.

\bibitem{SKL09}
M.~Stingl, M.~Ko\v{c}vara, and G.~Leugering.
\newblock A sequential convex semidefinite programming algorithm with an
  application to multiple-load free material optimization.
\newblock {\em SIAM Journal on Optimization}, 20(1):130--155, 2009.

\bibitem{WS15}
Alemseged~Gebrehiwot Weldeyesus and Mathias Stolpe.
\newblock A primal-dual interior point method for large-scale free material
  optimization.
\newblock {\em Computational Optimization and Applications}, 61(2):409--435,
  2015.

\bibitem{Xia15}
Yu~Xia.
\newblock Gradient methods and conic least-squares problems.
\newblock {\em Optimization Methods and Software}, 30(4):769--803, 2015.

\bibitem{ZKB97}
Jochem Zowe, Michal Ko\v{c}vara, and Martin~P. Bendsøe.
\newblock Free material optimization via mathematical programming.
\newblock {\em Mathematical Programming}, 79(1-3):445--466, 1997.

\end{thebibliography}
}

\end{document}